\documentclass[journal]{IEEEtran}

\ifCLASSINFOpdf
  \usepackage[pdftex]{graphicx}
  \graphicspath{{../figs/}}
\else
\fi
\usepackage{amsmath}
\usepackage{algorithmic}
\usepackage[dvipsnames]{xcolor}
\usepackage{hyperref}
\hypersetup{colorlinks=true,linkcolor=blue,urlcolor=NavyBlue}

\usepackage{tikz}
\usepackage{multirow, booktabs}
\usetikzlibrary{shapes,arrows}
\usepackage{amsfonts}
\usepackage{amssymb}
\usepackage{amsthm}
\usepackage[style=ieee, url=false, isbn=false, doi=false]{biblatex}
\bibliography{references}

\foreach \x in {a,...,z}{
  \expandafter\xdef\csname bf\x \endcsname{\noexpand\ensuremath{\noexpand\mathbf{\x}}}
  \expandafter\xdef\csname bm\x \endcsname{\noexpand\ensuremath{\noexpand\boldsymbol{\x}}}
}
\foreach \x in {A,...,Z}{
  \expandafter\xdef\csname bs\x \endcsname{\noexpand\ensuremath{\noexpand\boldsymbol{\x}}}
  \expandafter\xdef\csname bf\x \endcsname{\noexpand\ensuremath{\noexpand\mathbf{\x}}}
  \expandafter\xdef\csname bb\x \endcsname{\noexpand\ensuremath{\noexpand\mathbb{\x}}}
  \expandafter\xdef\csname ds\x \endcsname{\noexpand\ensuremath{\noexpand\mathds{\x}}}
  \expandafter\xdef\csname cal\x \endcsname{\noexpand\ensuremath{\noexpand\mathcal{\x}}}
}

\def\transp{\top}
\def\herm{\mathsf{H}}

\def\bmmu{\boldsymbol{\mu}}
\renewcommand{\phi}{\varphi}

\def\real{\mathrm{Re}\,}
\def\imag{\mathrm{Im}\,}

\newcommand{\invol}[2]{{#1}^{#2}}
\newcommand{\conj}[1]{#1^{\ast}}
\newcommand{\conjinvol}[2]{{#1}^{#2\ast}}
\newcommand{\ip}[1]{\left\langle #1 \right\rangle}

\renewcommand{\epsilon}{\varepsilon}
\newcommand{\pdiff}[2]{\frac{\partial #1}{\partial #2}}

\newcommand{\aR}[1]{#1_{\calR}}
\newcommand{\aH}[1]{#1_{\calH}}

\newcommand{\gradR}{\nabla_{\calR}}
\newcommand{\gradH}{\nabla_{\calH}}
\newcommand{\gradQ}{\nabla_{\bfq}}
\newcommand{\conjgradQ}{\nabla_{\conj{\bfq}}}
\newcommand{\conjgradH}{\nabla_{\conj{\calH}}}

\DeclareMathOperator*{\argmin}{arg\,min}
\DeclareMathOperator*{\dom}{dom}
\DeclareMathOperator*{\prox}{\textbf{prox}}
\DeclareMathOperator*{\diag}{diag}

\newtheorem{proposition}{Proposition}
\newtheorem{definition}{Definition}

\newtheorem{remark}{Remark}
\def\ie{\emph{i.e.}~}

\begin{document}
\urlstyle{tt}

\title{A general framework for\\ constrained convex quaternion optimization}
%
%
%

\author{Julien~Flamant,
  Sebastian~Miron,
  and~David~Brie
  \thanks{J. Flamant, S. Miron and D. Brie are with the Université de Lorraine, CNRS, CRAN, F-54000 Nancy, France. Corresponding author \url{julien.flamant@cnrs.fr}. Authors acknowledge funding support of CNRS and GDR ISIS through project OPENING.}}

\maketitle

\begin{abstract}
  This paper introduces a general framework for solving constrained convex quaternion optimization problems in the quaternion domain.
  To soundly derive these new results, the proposed approach leverages the recently developed generalized $\bbH\bbR$-calculus together with the equivalence between the original quaternion optimization problem and its augmented real-domain counterpart.
  This new framework simultaneously provides rigorous theoretical foundations as well as elegant, compact quaternion-domain formulations for optimization problems in quaternion variables.
  Our contributions are threefold: \emph{(i)} we introduce the general form for convex constrained optimization problems in quaternion variables, \emph{(ii)} we extend fundamental notions of convex optimization to the quaternion case, namely Lagrangian duality and optimality conditions, \emph{(iii)} we develop the quaternion alternating direction method of multipliers (Q-ADMM) as a general purpose quaternion optimization algorithm.
  The relevance of the proposed methodology is demonstrated by solving two typical examples of constrained convex quaternion optimization problems arising in signal processing.
  Our results open new avenues in the design, analysis and efficient implementation of quaternion-domain optimization procedures.
\end{abstract}

\begin{IEEEkeywords}
  quaternion convex optimization; widely affine constraints; optimality conditions; alternating direction method of multipliers;
\end{IEEEkeywords}
\IEEEpeerreviewmaketitle

\section{Introduction}
\IEEEPARstart{T}{he use} of quaternion representations is becoming prevalent in many fields, including robotics \cite{127239},
attitude control and estimation \cite{markley2000quaternion,kristiansen2005satellite}
polarized signal processing \cite{miron2006quaternion,flamant_timefrequency_2019,flamant_complete_2018}, rolling bearing fault diagnosis \cite{yi_quaternion_2017}, computer graphics \cite{shoemake1985animating}, among others.
Compared to conventional real and complex models, quaternion algebra permits unique insights into the physics and the geometry of the problem at hand, while preserving a mathematically sound framework.
This is particularly true in color imaging \cite{chen2015color,subakan2011quaternion,xu_vector_2015,chen_low-rank_2020,miao_low-rank_2020,miao_quaternion-based_2020,liu_constrained_2020,xia_penalty_nodate,zou_quaternion_2021,zou_quaternion_2016,chan_complex_2016,flamant_quaternion_2020,mizoguchi_hypercomplex_2018,wang2021robust,mizoguchi_hypercomplex_2019,barthelemy_color_2015}, where quaternion models allow to efficiently and naturally handle cross correlations between color channels, unlike standard real-valued image representations -- based either on processing each color channel independently or by stacking them in a single long vector.
Quaternion modeling enables improved performances in color imaging tasks, such as sparse representations, low-rank approximations or missing data completion, to mention only a few.
Moreover, quaternions provide elegant and compact algebraic representations which may result in a reduction of the number of model parameters.
For instance, it has been recently shown \cite{zhu2018quaternion,parcollet2019quaternion} that quaternion convolutional neural networks (QCNNs) achieve better performance than conventional real CNNs while using fewer parameters.

Most problems involving quaternion models can be cast as the minimization of a real-valued function of quaternion variables.
Unfortunately, quaternion-domain optimization faces immediately a major obstacle: quaternion cost functions being real-valued, they are not differentiable according to quaternion analysis \cite{sudbery1979quaternionic} -- just like real-valued functions of complex variables are not differentiable according to complex analysis \cite{lang2013complex}.
Thus, for a long time, the intrinsic quaternion nature of quaternion optimization problems has been disregarded by reformulating them as optimization problems over the real field.
This procedure, however, is not completely satisfying.
Indeed, one can legitimately wonder whether such a real-domain equivalent model could have been directly designed without the algebraic / geometric insights offered by quaternion algebra.
Thus, it appears  more natural and intellectually fulfilling to solve the optimization problem directly in its originally quaternion form, rather than resorting to an equivalent \emph{ad hoc} real reformulation specifically designed for the problem at hand.
In this sense, developing a general quaternion optimization framework would enable end-to-end quaternion procedures, going from relevant algebraic / geometric models to practical algorithms.
Such a framework would boost the emergence of full quaternionic methodologies for efficiently solving practical problems, especially when dealing with 3D or 4D data.


%
Recently, a crucial step towards quaternion-domain optimization has been made with the development of the theory of $\bbH\bbR$-calculus \cite{mandic2011quaternion,xu2015theory,xu_enabling_2015,xu2016optimization}.
This new framework establishes a complete set of differentiation rules, encouraging the systematic development of quaternion-domain algorithms.
The $\bbH\bbR$-calculus can be seen as the generalization to quaternions of the $\bbC\bbR$-calculus \cite{kreutz-delgado_complex_2009}, which has enabled the formulation of several important complex-domain algorithms \cite{sorber_unconstrained_2012,li_alternating_2015}.
The theory of $\bbH\bbR$-calculus has led to the development of multiple quaternion-domain algorithms, notably in adaptive filtering \cite{jahanchahi_class_2013,che_ujang_quaternion-valued_2011,jahanchahi_class_2014}, low-rank quaternion matrix and tensor completion \cite{chen_low-rank_2020,miao_quaternion-based_2020,miao_low-rank_2020} and quaternion neural networks \cite{liu_constrained_2020,xia_penalty_nodate}.

Motivated by the increasing use of quaternion models in various applications, this paper provides a general methodology dedicated to quaternion constrained convex optimization problems.
We restrict ourselves to the convex case as it allows  to establish strong optimality guarantees.

The proposed framework builds upon two key ingredients: the recent theory of generalized \bbH\bbR-derivatives, to compute derivatives of cost functions defined in terms of quaternion variables, and the explicit correspondence between the original quaternion optimization problem and its real augmented counterpart, to rigorously establish general mathematical properties.
This work is part of the current effort in signal and image processing toward the development of general tools for solving optimization problems formulated in quaternion variables \cite{qi_quaternion_2021,he2021quaternion}.
%

%
More precisely, the use of systematic and explicit equivalencies between the original optimization problem in quaternion variables and its real-augmented domain counterpart allows us to derive once and for all fundamental results for  the general form of a convex constrained optimization problem in quaternion variables.
It permits to naturally extend elementary notions of convex optimization to the quaternion case, namely Lagrangian duality and optimality conditions.
Moreover, we develop the quaternion alternating direction method of multipliers (Q-ADMM) as a versatile quaternion optimization algorithm.
In essence, the real-augmented equivalent problem is used as a bridge to establish general full quaternion-domain optimization tools.
This methodology further reveals subtle properties of quaternion-domain optimization that would have been hindered otherwise.
This includes the \emph{widely affine} equality constraints that naturally arise in constrained quaternion problems, together with the quaternion nature of Lagrange dual variables associated with equality constraints.

The paper is organized as follows.
In Section \ref{sec:Preliminaries} we review the necessary material regarding quaternion variables, their different representations and discuss the general affine constraint in the quaternion domain.
Section \ref{sec:generalizedHR} describes generalized $\bbH\bbR$-calculus and its particular properties in the case of quaternion cost functions.
Section \ref{sec:convexoptimH} introduces the main theoretical tools for quaternion convex optimization problems, including definitions and optimality conditions.
Section \ref{sec:quaternionADMM} develops Q-ADMM in its general form to solve convex quaternion optimization problems.
Finally, we illustrate in Section \ref{sec:examples} the relevance of the proposed methodology by a detailed study of two examples of constrained quaternion optimization problems inspired by the existing signal processing literature. Section \ref{sec:conclusion} presents concluding remarks and perspectives.

\section{Preliminaries}
\label{sec:Preliminaries}

\subsection{Quaternion algebra}
The set of quaternions $\bbH$ defines a 4-dimensional normed division algebra over the real numbers $\bbR$.
It has canonical basis $\lbrace 1, \bmi, \bmj, \bmk\rbrace$, where $\bmi, \bmj, \bmk$ are imaginary units such that
\begin{equation}
  \bmi^2 = \bmj^2 = \bmk^2 = \bmi\bmj\bmk = -1, \quad \bmi\bmj = - \bmj \bmi, \bmi\bmj = \bmk\:.
\end{equation}
These relations imply in particular that quaternion multiplication is noncommutative, meaning that for $q, p \in \bbH$, one has $qp \neq pq$ in general.
Any quaternion $q \in \bbH$ can be written as
\begin{equation}
  q = q_a + \bmi q_b + \bmj q_c + \bmk q_d\:,
\end{equation}
where  $q_a, q_b, q_c, q_d \in \bbR$ are the components of $q$.
The real part of $q$ is $\real q = q_a$ whereas its imaginary part is $\imag q = \bmi q_b + \bmj q_c + \bmk q_d$.
A quaternion $q$ is said to be purely imaginary (or simply, pure) if $\real q = 0$.
The quaternion conjugate of $q$ is denoted by $\conj{q} = \real q - \imag q$ and acts on the product of two quaternions as $\conj{(pq)} = \conj{q}\conj{p}$.
The modulus of $q$ is $\vert q \vert = \sqrt{q\conj{q}} = \sqrt{\conj{q}q} = \sqrt{q_a^2 + q_b^2 + q_c^2 + q_d^2}$.
Any non-zero quaternion $q$ has an inverse $q^{-1} = \conj{q}/ \vert q \vert^2$.
The inverse of the product of two quaternions is $(pq)^{-1} = q^{-1}p^{-1}$.
Given a nonzero quaternion $\mu \in \bbH$, the transformation
\begin{equation}\label{eq:transformationRotation}
  \invol{q}{\mu} \triangleq \mu q \mu^{-1} = \frac{1}{\vert \mu \vert^2}\mu q \mu^*
\end{equation}
describes a three-dimensional rotation of the quaternion $q$.
In particular it satisfies the following properties
\begin{equation}
  \conjinvol{q}{\mu} \triangleq  \invol{(\conj{q})}{\mu}= \conj{(\invol{q}{\mu})}, \quad \invol{(pq)}{\mu} =  \invol{p}{\mu} \invol{q}{\mu}\:.
\end{equation}
Pure unit quaternions such as $\bmi, \bmj, \bmk$ (and more generally, any $\bmmu$ such that $\bmmu^2=-1$) play a special role and will be denoted in bold italic letters.
In this case the transformation \eqref{eq:transformationRotation} becomes an involution $\invol{q}{\bmmu} = -\bmmu q \bmmu$.
In particular
\begin{align}
  \invol{q}{\bmi} & = - \bmi q \bmi = q_a + \bmi q_b - \bmj q_c - \bmk q_d\:, \\
  \invol{q}{\bmj} & = - \bmj q \bmj = q_a - \bmi q_b + \bmj q_c - \bmk q_d\:, \\
  \invol{q}{\bmk} & = - \bmk q \bmk = q_a - \bmi q_b - \bmj q_c + \bmk q_d\:.
\end{align}
It follows that the components of $q$ can be directly expressed as a function of $q$ and its canonical involutions $\invol{q}{\bmi},\invol{q}{\bmj}, \invol{q}{\bmk}$ as
\begin{equation}
  \begin{split}
    q_a &= \frac{1}{4}\left(q + \invol{q}{\bmi} + \invol{q}{\bmj} + \invol{q}{\bmk}\right) \,,\\ q_b &= -\frac{\bmi}{4}\left(q + \invol{q}{\bmi} -\invol{q}{\bmj} - \invol{q}{\bmk}\right)\,,\\
    q_c &= -\frac{\bmj}{4}\left(q - \invol{q}{\bmi} + \invol{q}{\bmj} - \invol{q}{\bmk}\right) \,,\\
    q_d &= -\frac{\bmk}{4}\left(q - \invol{q}{\bmi} -\invol{q}{\bmj} + \invol{q}{\bmk}\right)\,.
  \end{split}
  \label{eq:involution2real}
\end{equation}

Any quaternion vector $\bfq \in \bbH^n$ can be written as $\bfq = \bfq_a + \bmi \bfq_b + \bmj \bfq_c + \bmk \bfq_d$, where $\bfq_a, \bfq_b, \bfq_c, \bfq_d \in \bbR^n$ are its components.
Similarly, any quaternion matrix $\bfA \in \bbH^{m \times n}$ can be expressed as $\bfA = \bfA_a + \bmi\bfA_b + \bmj\bfA_c + \bmk\bfA_d$ with $\bfA_a, \bfA_b, \bfA_c, \bfA_d \in \bbR^{m\times n}$.
The transpose of quaternion matrix $\bfA$ is denoted by $\bfA^\transp$.
Its conjugate transpose (or Hermitian) is $\bfA^\herm \triangleq (\conj{\bfA})^\transp = \conj{(\bfA^\transp)}$.
Unless otherwise stated, quaternion rotations or involutions are always applied entry-wise.
Note that quaternion matrix product requires special attention due to quaternion noncommutativity: that is for $\bfA \in \bbH^{m\times n}, \bfB \in \bbH^{n\times p}$, the $(i,j)$-th entry of $\bfA\bfB$ reads
\begin{equation}
  (\bfA\bfB)_{ij} = \sum_{k=1}^n A_{ik}B_{kj} \neq \sum_{k=1}^n B_{kj}A_{ik}\:.
\end{equation}
This implies notably that $(\bfA\bfB)^\transp \neq \bfB^\transp\bfA^\transp$ and $\conj{(\bfA\bfB)} \neq \conj{\bfA}\conj{\bfB}$ in general whereas $(\bfA\bfB)^\herm = \bfB^\herm\bfA^\herm$ always holds.
For more details on quaternions and quaternion linear algebra we refer the reader to \cite{ward2012quaternions, zhang_quaternions_1997} and references therein.

\subsection{Representation of quaternion vectors}
\label{sub:representationQuaternionVectors}
Quaternion vectors can be represented in three equivalent ways.
Let $\bfq \in \bbH^n$ and let us introduce
\begin{align}
  \calR & = \left\lbrace (\bfq_a^\transp, \bfq_b^\transp, \bfq_c^\transp, \bfq_d^\transp)^\transp \in \bbR^{4n} \mid \bfq \in \bbH^n  \right\rbrace\,,                                                    \\
  \calH & = \left\lbrace \left(\bfq^\transp, {\invol{\bfq}{\bmi}}^\transp, {\invol{\bfq}{\bmj}}^\transp, {\invol{\bfq}{\bmk}}^\transp\right)^\transp \in \bbH^{4n} \mid \bfq \in \bbH^n  \right\rbrace\:.
\end{align}
By definition, there exists a one-to-one mapping between each set $\bbH^n, \calR$ and $\calH$.
The set $\calR$ defines the \emph{augmented real representation} of $\bfq \in \bbH^n$ and can be identified with $\bbR^{4n}$.
We denote the augmented real vector by $\aR{\bfq}$.
The set $\calH$ defines the \emph{augmented quaternion representation} of $\bfq \in \bbH^n$ by making use of the three canonical involutions.
Importantly, $\calH$ is only a subset of $\bbH^{4n}$, \ie $\calH \subset \bbH^{4n}$.
We denote the augmented quaternion vector by $\aH{\bfq}$.

Expressions \eqref{eq:involution2real} show that $\aR{\bfq}$ and  $\aH{\bfq}$ are linked by the linear relationship
\begin{equation}
  \aH{\bfq} = \bfJ_n \aR{\bfq},\:
  \bfJ_n \triangleq \begin{bmatrix}
    \bfI_n & \bmi\bfI_n  & \bmj\bfI_n  & \bmk \bfI_n  \\
    \bfI_n & \bmi\bfI_n  & -\bmj\bfI_n & - \bmk\bfI_n \\
    \bfI_n & -\bmi\bfI_n & \bmj\bfI_n  & -\bmk \bfI_n \\
    \bfI_n & -\bmi\bfI_n & -\bmj\bfI_n & \bmk \bfI_n
  \end{bmatrix}.
\end{equation}
It can be shown that $\bfJ_n \in \bbH^{4n \times 4n}$ is invertible, with inverse $\bfJ^{-1}_n = \frac{1}{4}\bfJ^\herm_n$ so that conversely
\begin{equation}
  \aR{\bfq} = \frac{1}{4}\bfJ^\herm_n\aH{\bfq}\:.
\end{equation}{}
Now, equip each representation $\bbH^n$, $\calR$ and $\calH$ with the following real-valued inner products:
\begin{align}
  \ip{\bfq, \bfp}_{\bbH^n}          & \triangleq \real\left(\bfq^\herm\bfp \right)\,, \label{eq:innerproductHn} \\
  \ip{\aR{\bfq}, \aR{\bfp}}_{\calR} & \triangleq \aR{\bfq}^\transp \aR{\bfp}\,,                                 \\
  \ip{\aH{\bfq}, \aH{\bfp}}_{\calH} & \triangleq \frac{1}{4}\aH{\bfq}^\herm \aH{\bfp}\:.
\end{align}
Proposition \ref{prop:conservationInnerProduct} below shows that inner products are preserved from one representation to another.
\begin{proposition}\label{prop:conservationInnerProduct}
  Given $\bfq, \bfp \in \bbH^n$, the following equalities hold:
  \begin{equation}
    \ip{\bfq, \bfp}_{\bbH^n}  = \ip{\aR{\bfq}, \aR{\bfp}}_{\calR} = \ip{\aH{\bfq}, \aH{\bfp}}_{\calH}\:.
  \end{equation}
\end{proposition}
\begin{proof}
  Let $\bfq = \bfq_a + \bmi\bfq_b + \bmj\bfq_c + \bmk\bfq_d$ and $\bfp = \bfp_a + \bmi\bfp_b + \bmj\bfp_c + \bmk\bfp_d$ be vectors of $\bbH^n$.
  Using rules of quaternion calculus, a direct calculation shows that
  \begin{equation*}
    \real\left(\bfq^\herm \bfp\right) = \bfq_a^\transp\bfp_a + \bfq_b^\transp\bfp_b + \bfq_c^\transp\bfp_c + \bfq_d^\transp\bfp_d =  \aR{\bfq}^\transp \aR{\bfp}\:.
  \end{equation*}
  Then, by developing  $\aH{\bfq}^\herm \aH{\bfp}$ and using \eqref{eq:involution2real}, one gets
  \begin{align*}
    \aH{\bfq}^\herm \aH{\bfp} & = \bfq^\herm\bfp + {\invol{\bfq}{\bmi}}^\herm\invol{\bfp}{\bmi} + {\invol{\bfq}{\bmj}}^\herm\invol{\bfp}{\bmj} + {\invol{\bfq}{\bmk}}^\herm\invol{\bfp}{\bmk} \\
                              & = \bfq^\herm\bfp + \invol{\left(\bfq^\herm \bfp\right)}{\bmi} + \invol{\left(\bfq^\herm \bfp\right)}{\bmj} + \invol{\left(\bfq^\herm \bfp\right)}{\bmk}       \\
                              & = 4 \real\left(\bfq^\herm \bfp\right)
  \end{align*}
  which concludes the proof.
  This result can also be obtained using $\aH{\bfq}^\herm \aH{\bfp} = \aR{\bfq}^\herm\bfJ^\herm_n\bfJ_n \aR{\bfp} = 4\aR{\bfq}^\transp \aR{\bfp}$.
\end{proof}
\noindent It directly follows that induced norms are equal, such that $\Vert \bfq \Vert_{\bbH^n}  = \Vert \aR{\bfq} \Vert_{\calR} = \Vert \aH{\bfq} \Vert_{\calH}$.
Moreover, the norm induced by the real-valued inner product \eqref{eq:innerproductHn} is identical to the standard quaternion $2$-norm, since
\begin{equation}
  \Vert \bfq \Vert_{\bbH^n}^2 = \real\left(\bfq^\herm \bfq\right) = \Vert \bfq \Vert_2^2 \triangleq \sum_{i=1}^n\vert q_i\vert^2\:.
\end{equation}
For brevity, the $2$-norm notation will thus be used hereafter independently from the representation $\bbH^n$, $\calR$ or $\calH$.
In addition, whenever there is a risk of ambiguity, we add a subscript to quantities $\calR$ and $\calH$ in order to indicate the underlying dimension.

\subsection{Affine and linear constraints}
\label{sub:affine}
Affine equality constraints are ubiquitous in standard, real-domain constrained convex optimization.
To perform optimization directly in the quaternion domain, a key step is to investigate how such constraints translate to quaternions.
Let $\bfq \in \bbH^n$ be the variable of interest.
Using its augmented real representation, the most general affine constraint reads
\begin{equation}\label{eq:affineConstraintRealAugmented}
  \aR{\bfA} \bfq_{\calR} = \bfb_{\calR}, \quad \aR{\bfA} \in \bbR^{4p\times 4n}\:, \bfb_{\calR} \in \bbR^{4p} \tag{$\calA_{\calR}$}\:.
\end{equation}
The constraint is said to be linear if $\bfb_{\calR_p} =0$.
Our goal is to find equivalent formulations of \eqref{eq:affineConstraintRealAugmented} in the augmented quaternion domain $\calH$ and in $\bbH^n$, respectively.
Using the linear mapping between $\calH$ and $\calR$, one gets
\begin{align}
  \protect\eqref{eq:affineConstraintRealAugmented} & \Longleftrightarrow  \frac{1}{4}\bfJ_p\aR{\bfA} \bfJ^\herm_n \bfq_{\calH} = \bfJ_p\bfb_{\calR} \label{eq:affineEquationUnsimplified}\:, \\
                                                   & \Longleftrightarrow \aH{\bfA} \bfq_{\calH} = \bfb_{\calH} \label{eq:affineEquationHnHp}\tag{$\calA_{\calH}$}\:.
\end{align}
The constraint remains affine in the augmented quaternion domain, yet the matrix $\aH{\bfA} \in \bbH^{4p\times 4n}$ exhibits a specific band structure.
To see this, let us describe $\aR{\bfA}$ in terms of $p\times n$ real-valued blocks:
\begin{equation}
  \aR{\bfA} \triangleq
  \begin{bmatrix}
    \bfA_{11}^{\calR} & \bfA_{12}^{\calR} & \bfA_{13}^{\calR} & \bfA_{14}^{\calR}    \\
    \bfA_{21}^{\calR} & \bfA_{22}^{\calR} & \bfA_{23}^{\calR} & \bfA_{24}^{\calR}    \\
    \bfA_{31}^{\calR} & \bfA_{32}^{\calR} & \bfA_{33}^{\calR} & \bfA_{34}^{\calR}    \\
    \bfA_{41}^{\calR} & \bfA_{42}^{\calR} & \bfA_{43}^{\calR} & \bfA_{44}^{\calR}\:.
  \end{bmatrix}
\end{equation}{}
Then, computations yield
\begin{align}
  \bfJ_p \aR{\bfA}
   & = \begin{bmatrix}
    \tilde{\bfA}_{\cdot 1}        & \tilde{\bfA}_{\cdot 2}        & \tilde{\bfA}_{\cdot 3}        & \tilde{\bfA}_{\cdot 4}        \\
    \tilde{\bfA}_{\cdot 1}^{\bmi} & \tilde{\bfA}_{\cdot 2}^{\bmi} & \tilde{\bfA}_{\cdot 3}^{\bmi} & \tilde{\bfA}_{\cdot 4}^{\bmi} \\
    \tilde{\bfA}_{\cdot 1}^{\bmj} & \tilde{\bfA}_{\cdot 2}^{\bmj} & \tilde{\bfA}_{\cdot 3}^{\bmj} & \tilde{\bfA}_{\cdot 4}^{\bmj} \\
    \tilde{\bfA}_{\cdot 1}^{\bmk} & \tilde{\bfA}_{\cdot 2}^{\bmk} & \tilde{\bfA}_{\cdot 3}^{\bmk} & \tilde{\bfA}_{\cdot 4}^{\bmk}
  \end{bmatrix}
\end{align}{}
where $\tilde{\bfA}_{\cdot j} = \bfA_{1j}^{\calR} + \bmi \bfA_{2j}^{\calR} + \bmj\bfA_{3j}^{\calR} + \bmk \bfA_{4j}^{\calR} \in \bbH^{p \times n}$ for $j = 1, 2, 3, 4$.
Next, define four quaternion matrices $\bfA_1, \bfA_2, \bfA_3, \bfA_4 \in \bbH^{p\times n}$ such that
\begin{align}
  {\bfA}_1 & = \frac{1}{4} \left(\tilde{\bfA}_{\cdot 1} - \tilde{\bfA}_{\cdot 2}\bmi - \tilde{\bfA}_{\cdot 3}\bmj - \tilde{\bfA}_{\cdot 4}\bmk\right)     \\
  {\bfA}_2 & =  \frac{1}{4} \left(\tilde{\bfA}_{\cdot 1} - \tilde{\bfA}_{\cdot 2}\bmi + \tilde{\bfA}_{\cdot 3}\bmj + \tilde{\bfA}_{\cdot 4}\bmk\right)    \\
  {\bfA}_3 & =  \frac{1}{4} \left(\tilde{\bfA}_{\cdot 1} + \tilde{\bfA}_{\cdot 2}\bmi - \tilde{\bfA}_{\cdot 3}\bmj + \tilde{\bfA}_{\cdot 4}\bmk\right)    \\
  {\bfA}_4 & =  \frac{1}{4} \left(\tilde{\bfA}_{\cdot 1} + \tilde{\bfA}_{\cdot 2}\bmi + \tilde{\bfA}_{\cdot 3}\bmj - \tilde{\bfA}_{\cdot 4}\bmk\:\right).
\end{align}
Then, the matrix $\aH{\bfA} \triangleq \frac{1}{4}\bfJ_p \aR{\bfA} \bfJ^\herm_n$ explicitly reads:
\begin{align}
  \aH{\bfA} = \begin{bmatrix}
    {\bfA}_1               & {\bfA}_2               & {\bfA}_3               & {\bfA}_4                 \\
    \invol{{\bfA}}{\bmi}_2 & \invol{{\bfA}}{\bmi}_1 & \invol{{\bfA}}{\bmi}_4 & \invol{{\bfA}}{\bmi}_3   \\
    \invol{{\bfA}}{\bmj}_3 & \invol{{\bfA}}{\bmj}_4 & \invol{{\bfA}}{\bmj}_1 & \invol{{\bfA}}{\bmj}_2   \\
    \invol{{\bfA}}{\bmk}_4 & \invol{{\bfA}}{\bmk}_3 & \invol{{\bfA}}{\bmk}_2 & \invol{{\bfA}}{\bmk}_1\:
  \end{bmatrix}\label{eq:structuredBandMatrixA}
\end{align}

A careful inspection of the expressions above shows that there exists a one-to-one mapping between quaternion matrices $\bfA_1, \bfA_2, \bfA_3, \bfA_4$ and the real matrix blocks $\bfA^{\calR}_{ij}$ that define $\aR{\bfA}$.
To obtain the corresponding constraint in the quaternion domain, remark that each line of the system of equations \eqref{eq:affineEquationHnHp} corresponds to a block of $p$-equations.
Writing \eqref{eq:affineEquationHnHp} explictly, we see that the second, third and fourth $p$-blocks are simply canonical involutions of the first block of equations.
As a result, this shows that \eqref{eq:affineEquationHnHp} is equivalent to the following \emph{widely affine} quaternion constraint
\begin{equation}
  \bfA_1 \bfq + \bfA_2 \invol{\bfq}{\bmi} + \bfA_3 \invol{\bfq}{\bmj} +  \bfA_4 \invol{\bfq}{\bmk} = \bfb\:.\tag{$\calA$}\label{eq:affineConstraintQuaternion}
\end{equation}
The term \emph{widely} refers to the fact that $\bfq$ and its three canonical involutions $\invol{\bfq}{\bmi} , \invol{\bfq}{\bmj}$ and $\invol{\bfq}{\bmk}$ are necessary to describe the most general affine constraint in the corresponding augmented real space.
In particular, strictly affine quaternion constraints like $\bfA \bfq = \bfb$ -- which are commonly found in the existing literature -- are only a special case of widely affine constraints described by \eqref{eq:affineConstraintQuaternion}.
Focusing on this special case, stricly affine quaternion constraints $\bfA\bfq = \bfb$ impose a peculiar structure on the associated real matrix $\aR{\bfA}$.
The matrix $\aH{\bfA}$ becomes diagonal, and letting $\bfA = \bfA_a + \bfA_b \bmi + \bfA_c \bmj + \bfA_d \bmk$, with $\bfA_a, \bfA_b, \bfA_c, \bfA_d \in \bbR^{p\times n}$ one obtains
\begin{align}
  \aR{\bfA} & = \frac{1}{4}\bfJ_p^\herm\diag(\bfA, \invol{\bfA}{\bmi},\invol{\bfA}{\bmj},\invol{\bfA}{\bmk})  \bfJ_n \:, \\
            & = \begin{bmatrix}
    \bfA_a & -\bfA_b & -\bfA_c  & - \bfA_d \\
    \bfA_b & \bfA_a  & - \bfA_d & \bfA_c   \\
    \bfA_c & \bfA_d  & \bfA_a   & -\bfA_b  \\
    \bfA_d & -\bfA_c & \bfA_b   & \bfA_a
  \end{bmatrix}\:,
\end{align}
hence $\aR{\bfA}$ is a real structured matrix of size $4p\times 4n$.

\section{Optimization of real-valued functions\\ of quaternion variables}
\label{sec:generalizedHR}

Given a real-valued function of quaternion variables (e.g. a cost function), is it possible to define quaternion derivatives and if so, how can we compute them?
One key obstacle lies in the non-analytic nature of real-valued functions: this means, in particular, that such functions are not quaternion differentiable \cite{sudbery1979quaternionic,watson2003generalized} and that other strategies need to be developed.

First, a pseudo-derivative approach can be used by treating a function $f$ of the variable $\bfq \in \bbH^n$ as a function of its four real components $\bfq_a, \bfq_b, \bfq_c, \bfq_c$ -- however, as the compactness of quaternion expressions is lost, such an approach may  require tedious and cumbersome computations.
Alternatively, the recent advent of (generalized) $\bbH\bbR$-calculus \cite{mandic2011quaternion,xu_enabling_2015} paved the way to efficient computation of quaternion-domain derivatives.
It provides a complete framework generalizing the $\bbC\bbR$-calculus \cite{kreutz-delgado_complex_2009} of complex-valued optimization to the case of quaternion functions.
Generalized $\bbH\bbR$-calculus is one of the key ingredients of the proposed framework for constrained quaternion optimization.
This section covers the fundamental definitions and properties, focusing on practical aspects.
For detailed proofs and discussions, we refer the reader to the pioneering papers \cite{xu2015theory,xu_enabling_2015,xu2016optimization}.

\subsection{Generalized $\bbH\bbR$-derivatives for cost functions}

We first consider the simpler case of a univariate function $f: \bbH \rightarrow \bbR$.
The function $f$ is said to be real-differentiable \cite{sudbery_quaternionic_1979} if the function $f(q) = f(q_a, q_b, q_c, q_d) = f(\aR{q})$ is differentiable with respect to the real variables $q_a, q_b, q_c, q_d$.
Generalized $\bbH\bbR$ (GHR)-derivatives  are defined in terms of standard real-domain derivatives as follows.
\begin{definition}[Generalized $\bbH\bbR$-derivatives \cite{xu_enabling_2015}]
  Let $\mu$ be a nonzero quaternion.
  The GHR derivatives of a real-differentiable $f:\bbH\rightarrow \bbR$ with respect to $\invol{q}{\mu}$ and $\conjinvol{q}{\mu}$ are
  \begin{align}
    \pdiff{f}{\invol{q}{\mu}}     & = \frac{1}{4}\left(\pdiff{f}{q_a} -\pdiff{f}{q_b}\invol{\bmi}{\mu} -\pdiff{f}{q_c}\invol{\bmj}{\mu} - \pdiff{f}{q_d}\invol{\bmk}{\mu}\right)\label{eq:GHRderivativeUnivariate}      \\
    \pdiff{f}{\conjinvol{q}{\mu}} & = \frac{1}{4}\left(\pdiff{f}{q_a} + \pdiff{f}{q_b}\invol{\bmi}{\mu} +\pdiff{f}{q_c}\invol{\bmj}{\mu} + \pdiff{f}{q_d}\invol{\bmk}{\mu}\right)\label{eq:GHRderivativeConjUnivariate}
  \end{align}
\end{definition}
The term \emph{generalized} refers to the use of an arbitrary quaternion rotation encoded by $\mu \neq 0$ in expressions \eqref{eq:GHRderivativeUnivariate}--\eqref{eq:GHRderivativeConjUnivariate}.
This is necessary to ensure that GHR calculus can be equipped with product and chain rules (see \cite{xu_enabling_2015} for details and further properties).
Since $f$ is real-valued, its GHR derivatives enjoy several nice properties, such as the conjugate rule
\begin{equation}
  \left(\pdiff{f}{\invol{q}{\mu}}\right)^\ast  = \pdiff{f}{\conjinvol{q}{\mu}}, \quad     \left(\pdiff{f}{\conjinvol{q}{\mu}}\right)^\ast  = \pdiff{f}{\invol{q}{\mu}}\:,\label{eq:conjugateRule}
\end{equation}
together with a special instance of the rotation rule \cite{xu_enabling_2015}
\begin{equation}
  \pdiff{f}{\invol{q}{\nu}} = \invol{\left(\pdiff{f}{q}\right)}{\nu}, \quad \pdiff{f}{\conjinvol{q}{\nu}} = \invol{\left(\pdiff{f}{\conj{q}}\right)}{\nu}\:.
\end{equation}
for $\nu \in \lbrace 1, \bmi, \bmj, \bmk\rbrace$.

\subsection{Quaternion gradient and stationary points}
\label{sub:quaternionGradients}

Consider now a real-valued function $f: \bbH^n \rightarrow \bbR$ of the quaternion vector variable $\bfq = (q_1, q_2, \ldots q_n)^\transp \in \bbH^n$.
We assume that $f$ is real-differentiable, that is, is real-differentiable with respect to each vector component $q_i$, $i=1, 2, \ldots, n$.
The $\mu$-gradient operator and $\mu$-conjugated gradient operators are defined in terms of GHR derivatives as follows \cite{xu2016optimization}:
\begin{align}
  \nabla_{\invol{\bfq}{\mu}} f     & \triangleq  \left(\pdiff{f}{\invol{q_1}{\mu}}, \pdiff{f}{\invol{{q_2}}{\mu}}, \ldots, \pdiff{f}{\invol{q_n}{\mu}} \right)^\transp\in \bbH^{n},           \\
  \nabla_{\conjinvol{\bfq}{\mu}} f & \triangleq \left(\pdiff{f}{\conjinvol{q_1}{\mu}}, \pdiff{f}{\conjinvol{q_2}{\mu}}, \ldots, \pdiff{f}{\conjinvol{q_n}{\mu}} \right)^\transp \in \bbH^{n}.
\end{align}
Remark immediatly that since $f$ is real-valued, the conjugate rule \eqref{eq:conjugateRule} implies that $\nabla_{\invol{\bfq}{\mu}} f = \left(\nabla_{\conjinvol{\bfq}{\mu}} f\right)^*$.
When $\mu = 1$, we simply call $\gradQ f$ (resp. $\conjgradQ f$) the quaternion gradient of $f$ (resp. conjugated quaternion gradient of $f$).
Choosing the canonical involutions $\mu \in \lbrace 1, \bmi, \bmj, \bmk\rbrace$, we define the augmented quaternion gradient and conjugated augmented quaternion gradient as
\begin{equation}
  \gradH f \triangleq \begin{pmatrix} \nabla_{\bfq} f \\ \nabla_{\invol{\bfq}{\bmi}} f\\  \nabla_{\invol{\bfq}{\bmj}} f\\ \nabla_{\invol{\bfq}{\bmk}} f\end{pmatrix}, \:
  \conjgradH f \triangleq \begin{pmatrix} \nabla_{\conj{\bfq}} f \\ \nabla_{\conjinvol{\bfq}{\bmi}} f\\  \nabla_{\conjinvol{\bfq}{\bmj}} f\\ \nabla_{\conjinvol{\bfq}{\bmk}} f\end{pmatrix} \in \bbH^{4n}\:.
\end{equation}
Introducing the (standard) augmented real gradient  operator as $\gradR \triangleq (\nabla_{\bfq_a}^\transp,\nabla_{\bfq_b}^\transp, \nabla_{\bfq_c}^\transp, \nabla_{\bfq_d}^\transp)^\transp$ and exploiting the definition of generalized $\bbH\bbR$ derivatives \eqref{eq:GHRderivativeUnivariate}--\eqref{eq:GHRderivativeConjUnivariate}, one obtains
\begin{align}
  \gradH f & = \frac{1}{4} \begin{bmatrix}
    \bfI_n & -\bmi\bfI_n & -\bmj\bfI_n & -\bmk \bfI_n \\
    \bfI_n & -\bmi\bfI_n & \bmj\bfI_n  & \bmk\bfI_n   \\
    \bfI_n & \bmi\bfI_n  & -\bmj\bfI_n & \bmk \bfI_n  \\
    \bfI_n & \bmi\bfI_n  & \bmj\bfI_n  & -\bmk \bfI_n \\
  \end{bmatrix}\begin{bmatrix}\nabla_{\bfq_a} f \\\nabla_{\bfq_b} f\\\nabla_{\bfq_c} f\\
    \nabla_{\bfq_d} f\end{bmatrix}\nonumber \\
           & = \frac{1}{4}\conj{\bfJ_n}\gradR f\:,\label{eq:gradHtogradR}
\end{align}
and
\begin{align}
  \conjgradH f & = \frac{1}{4} \begin{bmatrix}
    \bfI_n & \bmi\bfI_n  & \bmj\bfI_n  & \bmk \bfI_n  \\
    \bfI_n & \bmi\bfI_n  & -\bmj\bfI_n & -\bmk\bfI_n  \\
    \bfI_n & -\bmi\bfI_n & \bmj\bfI_n  & -\bmk \bfI_n \\
    \bfI_n & -\bmi\bfI_n & -\bmj\bfI_n & \bmk \bfI_n  \\
  \end{bmatrix}\begin{bmatrix}\nabla_{\bfq_a} f\\\nabla_{\bfq_b} f\\\nabla_{\bfq_c} f\\\nabla_{\bfq_d} f\end{bmatrix} \nonumber \\
               & = \frac{1}{4}\bfJ_n\gradR f\:.\label{eq:gradHconjtogradR}
\end{align}
In particular, the real augmented and conjugated augmented quaternion gradients are related by a simple linear transform:
\begin{equation}
  \gradR f = \bfJ^\herm_n\conjgradH f \:.\label{eq:gradRtogradHconj}
\end{equation}
This result is fundamental for the proposed quaternion optimization framework since it permits to switch from one representation of quaternion vectors to another while preserving gradient-related properties.
Notably, it allows to derive necessary and sufficient conditions for stationary points of real-valued functions of quaternion variables.
\begin{proposition}[Stationary points]
  Let $f : \bbH^n \rightarrow \bbR$ be real-differentiable and let $\mu \in \lbrace 1, \bmi, \bmj, \bmk\rbrace$.
  The vector $\bfq_\star \in \bbH^n$ is a stationary point of $f$ iff
  \begin{equation}
    \begin{split}
      \nabla_{\invol{\bfq}{\mu}} f (\bfq_\star) = 0 \Leftrightarrow \nabla_{\conjinvol{\bfq}{\mu}} f (\bfq_\star) = 0 \Leftrightarrow \gradR f (\aR{{\bfq_\star}}) = 0\\
      \Leftrightarrow \ \gradH f (\aH{{\bfq_\star}}) = 0 \Leftrightarrow \conjgradH f (\aH{{\bfq_\star}}) = 0\:.
    \end{split}
  \end{equation}
  \label{prop:stationaryPoints}
\end{proposition}
\begin{proof}
  Let $\bfq_\star \in \bbH^n$ and define $\aR{{\bfq_\star}}$ and $\aH{{\bfq_\star}}$ its augmented real and quaternion vectors.
  Suppose that $\gradR f(\aR{{\bfq_\star}}) = 0$.
  By Eqs. \eqref{eq:gradHtogradR}--\eqref{eq:gradHconjtogradR} one has
  $$\gradR f (\aR{{\bfq_\star}}) = 0 \Leftrightarrow \gradH f (\aH{{\bfq_\star}}) = 0 \Leftrightarrow \conjgradH f (\aH{{\bfq_\star}}) = 0\:.$$
  Clearly, by definition of $\gradH$ one has $ \gradH f (\aH{{\bfq_\star}}) = 0  \Rightarrow  \gradQ f (\bfq_\star) = 0$.
  Conversely, suppose that $\gradQ f (\bfq_\star) = 0$. Since $f$ is real-valued, one has $\nabla_{\invol{\bfq}{\mu}} f = \invol{(\nabla_{\bfq} f)}{\mu}$ for $\mu \in \lbrace 1, \bmi, \bmj, \bmk\rbrace$, so that $\gradQ f (\bfq_\star) = 0 \Rightarrow \gradH f (\aH{{\bfq_\star}}) = 0$.
  Similarly one shows that $\conjgradQ f (\bfq_\star) = 0 \Leftrightarrow \conjgradH f (\aH{{\bfq_\star}}) = 0$, which concludes the proof.
\end{proof}
Proposition \ref{prop:stationaryPoints} has a very important consequence: it states that optimization problems involving quaternion variables can be equivalently tackled in any representation: $\bbH^n$, $\calR$ or $\calH$.
This equivalence allows to move back-and-forth between the three representations and to benefit from the advantages of each.
This result is a cornerstone for the proposed framework for quaternion convex optimization detailed in the remaining of this paper.

\section{Convex optimization \\with quaternion variables}
\label{sec:convexoptimH}
This section starts by introducing the notion of convex sets and convex functions in the quaternion domain.
Then we introduce the most general form for a constrained quaternion convex problem by leveraging the equivalent augmented real optimization problem.
The notion of Lagrangian and duality are introduced next, which enables the formulation of two fundamental optimality conditions.
Some of these definitions may appear trivial to the reader familiar with the convex optimization field: yet, in our opinion, explicit and rigorous definitions are necessary to ensure the soundness of the proposed framework for quaternion convex optimization.

\subsection{Convex sets and convex functions}
Definitions of convexity for quaternion sets of $\bbH^n$ or for a real-valued function of quaternions variables $f:\bbH^n \rightarrow \bbR$ appear very close to the standard real case.
This is essentially due to the fact that convexity is intrisically a ``real'' property, so that convexity in the quaternion domain is inherited from convexity of the equivalent, real-augmented representation.\\

\noindent\textbf{Convex sets. }
Let $\calC \subset \bbH^n$ and define $\aR{\calC} = \lbrace (\bfq_a^\transp, \bfq_b^\transp, \bfq_c^\transp, \bfq_d^\transp)^\transp \in \bbR^{4n} \mid \bfq \in \calC\rbrace$ its augmented real representation.
We say that $\calC$ is a convex set (resp. cone, convex cone) of $\bbH^n$ if $\aR{\calC}$ is a convex set (resp. cone, convex cone) of $\bbR^{4n}$.
This leads to the following explicit definitions.
\begin{definition}[Convex set]
  A set $\calC \subset \bbH^n$ is convex if $\forall\:\bfp, \bfq \in \calC$ and any $\theta \in [0, 1]$, one has $\theta \bfp + (1-\theta) \bfq \in \calC$.
\end{definition}

A similar definition is possible for cones and convex cones of $\bbH^n$.
\begin{definition}[Cone and convex cone]
  A set $\calC \subset \bbH^n$ is a cone if $\forall \: \bfq \in \calC$ and $\theta \geq 0$, $\theta \bfq \in \calC$.
  A set $\calC$ is a convex cone if it is convex and a cone, which means that $\forall\:\bfp, \bfq \in \calC$ and any $\theta_1, \theta_2 \geq 0$ we have $\theta_1 \bfp + \theta_2 \bfq \in \calC$.
\end{definition}

\begin{remark}
  Given a convex set $\calC \subset \bbH^n$ (resp. convex cone),
  $$\aH{\calC} \triangleq \left\lbrace (\bfq, \invol{\bfq}{\bmi}, \invol{\bfq}{\bmj}, \invol{\bfq}{\bmk})^\transp \in \bbH^{4n} \mid \bfq \in \calC  \right\rbrace$$
  is a convex set (resp. convex cone) of $\bbH^{4n}$. The converse is also true.
\end{remark}
Remaining definitions such as convex hull, dual cone, etc. for the quaternion domain are omitted for brevity.
They can be obtained if desired, by proceeding analogously and exploiting equivalence with the augmented real representation.

\noindent\textbf{Convex functions. }
Similarly to the definition of convex sets of quaternions, the definition of convexity of real-valued functions of quaternion variables relies on convexity of the associated function in terms of augmented real-variables.

\begin{definition}[Convex function]
  A function $f: \bbH^n \rightarrow \bbR$ is convex if its domain $\dom f$ is convex and if for all $\bfp, \bfq \in \dom f$ and for $\theta \in [0, 1]$ one has
  \begin{equation}
    f\left(\theta \bfp + (1-\theta)\bfq\right) \leq \theta f(\bfp) + (1-\theta)f(\bfq)
  \end{equation}
  A function $f$ is strictly convex if the above inequality is strict whenever $\bfp \neq \bfq$ and $\theta \in (0, 1)$.
\end{definition}
Remark that if $f(\bfq)$ is convex, the function $f(\aH{\bfq})$ is also convex, and conversely.

In practice, supposing that $f$ is real-differentiable it is possible to characterize convexity in terms of quaternion gradients introduced in Section \ref{sub:quaternionGradients}.
\begin{proposition}[First-order characterization] \label{prop:firstOrderConvexCharacterization}
  Consider $f:\bbH^n \rightarrow \bbR$ a real-differentiable function such that $\dom f$ is convex. Then $f$ is convex if and only if $\forall\: \bfp, \bfq \in \dom f$
  \begin{align}
    f(\bfp)                            & \geq f(\bfq) + 4 \real \left(\conjgradQ f(\bfq)^\herm (\bfp-\bfq)\right) \\
    \Longleftrightarrow  f(\aH{\bfp})  & \geq f(\aH{\bfq}) + \conjgradH f(\aH{\bfq})^\herm (\aH{\bfp}-\aH{\bfq})  \\
    \Longleftrightarrow   f(\aR{\bfp}) & \geq f(\aR{\bfq}) + \gradR f(\aR{\bfq})^\transp(\aR{\bfp}-\aR{\bfq})\:.
  \end{align}
\end{proposition}
\begin{proof}
  Suppose $\dom f$ convex and $f$ real-differentiable.
  Then the usual convexity condition \cite[Chapter 3]{boyd_convex_2004} on $f$ reads in terms of real augmented variables
  \begin{equation*}
    \text{$f$ is convex} \Leftrightarrow \begin{cases}
      \forall\: \bfp, \bfq \in \dom f, \\
      f(\aR{\bfp}) \geq f(\aR{\bfq}) + \gradR f(\aR{\bfq})^\transp(\aR{\bfp}-\aR{\bfq})
    \end{cases}.
  \end{equation*}
  Using \eqref{eq:gradRtogradHconj}, the inner product term can be rewritten as
  \begin{align*}
    \gradR f(\aR{\bfq})^\transp(\aR{\bfp}-\aR{\bfq}) & = \gradR f(\aR{\bfq})^\herm(\aR{\bfp}-\aR{\bfq})                                        \\
                                                     & = \left(\bfJ^\herm \conjgradH f (\aH{\bfq}) \right)^\herm\bfJ^{-1}(\aH{\bfp}-\aH{\bfq}) \\
                                                     & = \conjgradH f (\aH{\bfq})^\herm (\aH{\bfp}-\aH{\bfq})\:,
  \end{align*}
  which yields the second equivalency.
  To obtain the result in $\bfq, \bfp$ coordinates, remark that $ \conjgradH f (\aH{\bfq})$ is the quaternion augmented vector of $\nabla_{\conjinvol{\bfq}{}}f(\bfq)$, so that by Proposition \ref{prop:conservationInnerProduct}
  \begin{align*}
    \conjgradH f (\aH{\bfq})^\herm (\aH{\bfp}-\aH{\bfq}) = 4 \real \left(\conjgradQ f(\bfq)^\herm (\bfp-\bfq)\right)
  \end{align*}
  which concludes the proof.
\end{proof}
\textbf{Example. }
Consider the function $f: \bbH^n \rightarrow \bbR$ defined by
\begin{equation}
  f(\bfq) = \Vert \bfP_1 \bfq + \bfP_2 \invol{\bfq}{\bmi} + \bfP_3 \invol{\bfq}{\bmj} +  \bfP_4 \invol{\bfq}{\bmk} - \bfb \Vert^2_{2}\:,
\end{equation}
where $\bfP_1, \bfP_2, \bfP_3,\bfP_4 \in \bbH^{p\times n}$ and $\bfb \in \bbH^p$ are arbitrary.
First, note that $\dom f = \bbH^n$ is convex.
From Proposition \ref{prop:conservationInnerProduct}, $f(\bfq) = f(\aR{\bfq}) = \Vert \aR{\bfP}\aR{\bfq} - \aR{\bfp}\Vert_{2}^2$; $f$ is a convex function of the augmented real variable $\aR{\bfq}$, so $f$ is convex in $\bfq$.

For brevity, properties of quaternion convex functions such as closedness, properness, etc. are omitted. They can be defined without difficulty just like above, by exploiting the augmented real representation.

\subsection{Convex problems}
\label{sub:convexProblems}

The most general form of quaternion constrained convex optimization problems consists in  the minimization of a convex function subject to inequality constraints defined by convex functions and to widely affine equality constraints.
Formally,
\begin{equation}
  \begin{split}
    \text{minimize } & f_0(\bfq)\\
    \text{subject to } & f_i(\bfq) \leq 0, \: i = 1, \ldots, m\\
    & \bfA_1\bfq + \bfA_2\invol{\bfq}{\bmi} + \bfA_3\invol{\bfq}{\bmj}+ \bfA_4\invol{\bfq}{\bmk} = \bfb
  \end{split}\tag{$P$}\label{eq:GenericOptimH}
\end{equation}
where $f_0, \ldots f_m: \bbH^n \rightarrow \bbR$ are real-valued convex functions and where $\bfA_1, \bfA_2, \bfA_3, \bfA_4 \in \bbH^{p\times n}$ and $\bfb \in \bbH^{p}$ encode $p$ quaternion widely affine equality constraints.
This particular type of equality constraints -- specific to quaternion algebra -- ensures that \eqref{eq:GenericOptimH} defines the most general form of constrained convex quaternion optimization problems.
Using results from Section \ref{sub:affine}, the problem \eqref{eq:GenericOptimH} can be equivalently rewritten in terms of the augmented quaternion variable $\aH{\bfq} \in \calH \subset \bbH^{4n}$ as follows:
\begin{equation}
  \begin{split}
    \text{minimize } & f_0(\aH{\bfq})\\
    \text{subject to } & f_i(\aH{\bfq}) \leq 0, \: i = 1, \ldots, m\\
    & \aH{\bfA}\aH{\bfq} = \aH{\bfb}
  \end{split}\:,\tag{$P_{\calH}$}\label{eq:GenericOptimaH}
\end{equation}
where $\aH{\bfA} \in \bbH^{4p\times 4n}$ is the structured quaternion matrix given by \eqref{eq:structuredBandMatrixA}.
Similarly, one can obtain the equivalent constrained convex problem in real augmented variables as
\begin{equation}
  \begin{split}
    \text{minimize } & f_0(\aR{\bfq})\\
    \text{subject to } & f_i(\aR{\bfq}) \leq 0, \: i = 1, \ldots, m\\
    & \aR{\bfA}\aR{\bfq} = \aR{\bfb}
  \end{split}\:,\tag{$P_{\calR}$}\label{eq:GenericOptimaR}
\end{equation}
which is a constrained convex problem in real augmented variables written in its most general form.

In the sequel, the equivalence between the three optimization problems \eqref{eq:GenericOptimH}, \eqref{eq:GenericOptimaH} and \eqref{eq:GenericOptimaR} is thoroughly exploited to construct a general constrained convex optimization framework directly in the quaternion domain.

\subsection{Lagrangian and duality}
As a first step, we exploit the equivalent real problem  \eqref{eq:GenericOptimaR} as a bridge to obtain the Lagrangian associated with the quaternion optimization problem \eqref{eq:GenericOptimH}.
The Lagrangian associated with the real equivalent problem \eqref{eq:GenericOptimaR} is the function $L: \calR_n \times \bbR^{m} \times \calR_p$ defined by
\begin{equation}
  \begin{split}
    &L(\aR{\bfq}, \boldsymbol{\nu}, \aR{\boldsymbol{\lambda}})\\
    &\triangleq  f_0(\aR{\bfq})  + \sum_{i = 1}^m \nu_i f_i(\aR{\bfq}) + \aR{\boldsymbol{\lambda}}^\top \left(\aR{\bfA}\aR{\bfq} - \aR{\bfb}\right)\:,\label{eq:LagrangianaR}
  \end{split}
\end{equation}
where $\boldsymbol{\nu} = (\nu_1, \nu_2, \ldots, \nu_m)^\transp \in \bbR^m$ is the dual variable associated to the $m$ inequality constraints and $\aR{\boldsymbol{\lambda}} \in \calR_p$ is the dual variable corresponding to the $4p$ real-augmented equality constraints $\aR{\bfA}\aR{\bfq} = \aR{\bfb}$.
To get the expression of the Lagrangian in terms of quaternion variables, we exploit the linear relation between the real and quaternion augmented quaternion representations $\calR$ and $\calH$ described in Sections \ref{sub:representationQuaternionVectors} together with the equivalence between affine constraints described in Section \ref{sub:affine}.
Using Proposition \ref{prop:conservationInnerProduct}, we get from \eqref{eq:LagrangianaR} the Lagrangian in augmented quaternion variables as
\begin{equation}
  \begin{split}
    &L(\aH{\bfq}, \boldsymbol{\nu}, \aH{\boldsymbol{\lambda}})\\
    &= f_0(\aH{\bfq})  + \sum_{i = 1}^m \nu_i f_i(\aH{\bfq}) + \frac{1}{4}\aH{\boldsymbol{\lambda}}^\herm \left(\aH{\bfA}\aH{\bfq} - \aH{\bfb}\right)\label{eq:LagrangianaH}\:.
  \end{split}
\end{equation}
Applying the same approach, one obtains the expression of the Lagrangian of the quaternion optimization problem \eqref{eq:GenericOptimH} as the function $L: \bbH^n \times \bbR^{m} \times \bbH^p$ defined by
\begin{equation}
  \begin{split}
    &L(\bfq, \boldsymbol{\nu}, \boldsymbol{\lambda}) \triangleq f_0(\bfq)  + \sum_{i = 1}^m \nu_i f_i(\bfq) \\&+ \real \left[\boldsymbol{\lambda}^\herm \left(\bfA_1\bfq + \bfA_2\invol{\bfq}{\bmi} + \bfA_3\invol{\bfq}{\bmj}+ \bfA_4\invol{\bfq}{\bmk} - \bfb\right)\right]\:.\label{eq:LagrangianH}
  \end{split}\
\end{equation}
\begin{remark}
  It can be easily checked by direct calculation that $L(\bfq, \boldsymbol{\nu}, \boldsymbol{\lambda}) = L(\aH{\bfq}, \boldsymbol{\nu}, \aH{\boldsymbol{\lambda}}) = L(\aR{\bfq}, \boldsymbol{\nu}, \aR{\boldsymbol{\lambda}})$, meaning that they represent the same quantity.
\end{remark}
\begin{remark}
  The Lagrange dual variable $\boldsymbol{\nu} \in \bbR^m$ associated to the inequality constraints is always a real-vector variable no matter which representation ($\bbH^n, \calR, \calH)$ is chosen, since the inequality constraints are defined by the real-valued functions $f_i$, $i=1, \ldots m$.
  The main difference between \eqref{eq:LagrangianaR}, \eqref{eq:LagrangianaH} and \eqref{eq:LagrangianH} lies in the way the affine equality constraints are handled, since $p$ quaternion equality constraints are equivalent to $4p$ real equality constraints.
  This explains why, in the expression of the quaternion Lagrangian \eqref{eq:LagrangianH}, the Lagrange dual variable $\boldsymbol{\lambda}$ associated to equality constraints is a $p$-dimensional quaternion vector.
\end{remark}

The definition of the quaternion Lagrangian \eqref{eq:LagrangianH} is a key step.
It allows to define quaternion-domain counterparts of fundamentals tools from Lagrangian duality in a  straightforward way.
For instance, let $\calD$ denote the domain of the problem \eqref{eq:GenericOptimH} such that $  \calD = \cap_{i=0}^m \dom f_i \cap \dom \calA$, where $\dom \calA$ denotes the domain of the widely affine constraint \eqref{eq:affineConstraintQuaternion}.
The quaternion dual function is defined as
\begin{equation}
  g(\boldsymbol{\nu}, \boldsymbol{\lambda}) \triangleq \inf_{\bfq \in \calD}  L(\bfq, \boldsymbol{\nu}, \boldsymbol{\lambda})\:.
\end{equation}
The dual quaternion Lagrange problem then reads
\begin{equation}
  \begin{split}
    \text{maximize } & g(\boldsymbol{\nu}, \boldsymbol{\lambda})\\
    \text{subject to } & \boldsymbol{\nu} \succeq 0
  \end{split}\label{eq:dualLagrangeOptimH}
\end{equation}
Just like with standard real-domain optimization, the dual Lagrange function yields lower bounds on the optimal value of the original problem \eqref{eq:GenericOptimH}, meaning that weak duality holds.
Conditions for strong duality are not given here.
They may be derived as well, by simple adaptation of the real case, see e.g. \cite[Section 5.2.3]{boyd_convex_2004}.

\subsection{Optimality conditions}
Exploiting the equivalence between the quaternion optimization problem \eqref{eq:GenericOptimH} and the real, augmented optimization problem \eqref{eq:GenericOptimaR} we derive two fundamental optimality conditions for \eqref{eq:GenericOptimH}.
To simplify the presentation, assume that the functions $f_i$, $i=0, 1, \ldots m$ are real-differentiable.

\textbf{Simple optimality condition.} Applying the usual optimality conditions for the equivalent real convex optimization problem (see \cite[Section 4.2.3]{boyd_convex_2004} for details) allows to derive a simple optimality condition for real-differentiable $f_0$.
Let $\calF$ denote the feasibility set
\begin{equation}
  \calF \triangleq \left\lbrace \bfq \left\vert \begin{array}{ll}
    f_i(\bfq) \leq 0, i=1, \ldots m \\
    \bfA_1\bfq + \bfA_2\invol{\bfq}{\bmi} + \bfA_3\invol{\bfq}{\bmj}+ \bfA_4\invol{\bfq}{\bmk} = \bfb
  \end{array}\right.\right\rbrace\:.
\end{equation}
Then by using the first order characterization of the convexity of $f_0$ given in Proposition \ref{prop:firstOrderConvexCharacterization}, one obtains the following necessary and sufficient condition: the vector $\tilde{\bfq}$ is optimal for the problem \eqref{eq:GenericOptimH} if and only if $\tilde{\bfq} \in \calF$ and
\begin{equation}
  \real \left(\conjgradQ f_0(\tilde{\bfq})^\herm (\bfr-\tilde{\bfq})\right) \geq 0 \text{ for all } \bfr \in \calF\:.
\end{equation}

\textbf{Karush-Kuhn-Tucker (KKT) conditions.}
Considering the convex quaternion optimization problem \eqref{eq:GenericOptimH}, sufficient optimality conditions known as KKT conditions can be derived from its real equivalent convex optimization problem.
\begin{proposition}[KKT conditions] Consider the constrained quaternion convex problem \eqref{eq:GenericOptimH} with quaternion Lagrangian $L(\bfq, \boldsymbol{\nu}, \boldsymbol{\lambda})$ given in \eqref{eq:LagrangianH}.
  Let $\tilde{\bfq} \in \bbH^n, \tilde{\boldsymbol{\nu}} \in \bbR^m,\tilde{\boldsymbol{\lambda}} \in \bbH^p$ be any points such that
  \begin{align}
     & f_i(\tilde{\bfq}) \leq 0   \qquad i = 1, \ldots, m \label{eq:KKT1}                                                                                      \\
     & \bfA_1\tilde{\bfq} + \bfA_2\invol{\tilde{\bfq}}{\bmi} + \bfA_3\invol{\tilde{\bfq}}{\bmj}+ \bfA_4\invol{\tilde{\bfq}}{\bmk} - \bfb = 0 \label{eq:KKT2}   \\
     & \tilde{\nu}_i \geq 0  \qquad  i = 1, \ldots, m                                                                                                          \\
     & \tilde{\nu}_i f_i(\tilde{\bfq}) =  0  \qquad  i = 1, \ldots, m                                                                                          \\
     & \conjgradQ L(\tilde{\bfq}, \tilde{\boldsymbol{\nu}}, \tilde{\boldsymbol{\lambda}}) = 0 \label{eq:KKTend}                                              &
  \end{align}
  then $\tilde{\bfq}$ and $(\tilde{\boldsymbol{\nu}}, \tilde{\boldsymbol{\lambda}})$ are primal and dual optimal, with zero-duality gap.
\end{proposition}
KKT conditions look almost the same as standard KKT conditions for real problems, except primal feasibility for equality constraints \eqref{eq:KKT2} and the Lagrangian stationarity condition \eqref{eq:KKTend}.
Proposition \ref{prop:StationaryLagrangian} below provides the explicit form of the stationary condition for the Lagrangian in quaternion variables.
\begin{proposition}\label{prop:StationaryLagrangian}
  Let $\tilde{\bfq} \in \bbH^n, \tilde{\boldsymbol{\nu}} \in \bbR^m,\tilde{\boldsymbol{\lambda}} \in \bbH^p$ such that they satisfy KKT conditions \eqref{eq:KKT1}--\eqref{eq:KKTend}.
  The stationarity condition \eqref{eq:KKTend} is explicitly given by
  \begin{equation}
    \begin{split}
      & \conjgradQ f_0(\tilde{\bfq}) + \sum_{i=1}^m \tilde{\nu}_i \conjgradQ f_i(\tilde{\bfq})\\ & + \frac{1}{4}\left[\bfA^\herm_1\tilde{\boldsymbol{\lambda}} + \invol{\left(\bfA^\herm_2\tilde{\boldsymbol{\lambda}}\right)}{\bmi} + \invol{\left(\bfA^\herm_3\tilde{\boldsymbol{\lambda}}\right)}{\bmj} + \invol{\left(\bfA^\herm_4\tilde{\boldsymbol{\lambda}}\right)}{\bmk}\right] = 0.\label{eq:StationaryLagrangianH}
    \end{split}
  \end{equation}
\end{proposition}
\begin{proof}
  Using Proposition \ref{prop:stationaryPoints}, the stationarity condition can be equivalently expressed as
  \begin{equation}
    \begin{split}
      \conjgradQ L(\tilde{\bfq}, \tilde{\boldsymbol{\nu}}, \tilde{\boldsymbol{\lambda}}) = 0 \Leftrightarrow \gradR  L(\aR{\tilde{\bfq}}, \tilde{\boldsymbol{\nu}}, \aR{\tilde{\boldsymbol{\lambda}}}) = 0 \\ \Leftrightarrow \conjgradH  L(\aH{\tilde{\bfq}}, \tilde{\boldsymbol{\nu}}, \aH{\tilde{\boldsymbol{\lambda}}}) = 0\:,
    \end{split}
  \end{equation}
  where $L(\aR{\tilde{\bfq}}, \tilde{\boldsymbol{\nu}}, \aR{\tilde{\boldsymbol{\lambda}}})$ and $L(\aH{\tilde{\bfq}}, \tilde{\boldsymbol{\nu}}, \aH{\tilde{\boldsymbol{\lambda}}})$ are given by \eqref{eq:LagrangianaR} and \eqref{eq:LagrangianaH}, respectively.
  A straightforward calculation gives explicitly the stationarity condition in the $\calR$ representation
  \begin{equation}
    \begin{split}
      &\gradR  L(\aR{\tilde{\bfq}}, \tilde{\boldsymbol{\nu}}, \aR{\tilde{\boldsymbol{\lambda}}}) =  0 \\
      &\Leftrightarrow \gradR f_0(\aR{\tilde{\bfq}}) + \sum_{i=1}^m \tilde{\nu}_i \gradR f_i(\aR{\tilde{\bfq}}) + \aR{\bfA}^\transp \aR{\tilde{\boldsymbol{\lambda}}} = 0\:.
    \end{split}
  \end{equation}
  Turning to the $\calH$-domain condition, we exploit the linear relationship between vectors of $\calR$ and $\calH$ as well as the relation \eqref{eq:gradRtogradHconj} between $\calR$ and $\calH$ gradients $\conjgradH f = \frac{1}{4}\bfJ_n\gradR f $.
  After simplification, one obtains the stationarity condition in $\calH$:
  \begin{equation}
    \begin{split}
      &\conjgradH  L(\aH{\tilde{\bfq}}, \tilde{\boldsymbol{\nu}}, \aH{\tilde{\boldsymbol{\lambda}}}) = 0\\
      &\Leftrightarrow \conjgradH f_0(\aH{\tilde{\bfq}}) + \sum_{i=1}^m \tilde{\nu}_i \conjgradH f_i(\aH{\tilde{\bfq}}) + \frac{1}{4}\aH{\bfA}^\herm\aH{\tilde{\boldsymbol{\lambda}}} = 0\:.\label{eq:proofStationaryLagrangianaH}
    \end{split}
  \end{equation}
  To obtain the desired stationarity condition \eqref{eq:StationaryLagrangianH} in $\bbH^n$, we keep the $n$-first rows of \eqref{eq:proofStationaryLagrangianaH} and compute explicitly the $n$-first blocks of the quaternion matrix product $\aH{\bfA}^\herm\aH{\tilde{\boldsymbol{\lambda}}}$.
\end{proof}

\section{Quaternion Alternating Direction Method of Multipliers}
\label{sec:quaternionADMM}

The quaternion-domain optimization framework introduced in previous sections permits an efficient and natural derivation of quaternion-domain algorithms from their existing real-domain counterparts.
The methodology is as follows: given a quaternion-domain optimization problem in variable $\bfq \in \bbH^n$, we find its real augmented domain equivalent in variable $\aR{\bfq} \in \bbR^{4n}$.
Then, one can pick any real-domain algorithm to solve the real-augmented problem; once the iterates for $\aR{\bfq}$ are found, they are converted into quaternion augmented domain $\calH$.
Finally the quaternion-domain algorithm is obtained by considering only the first
$n$ entries of $\aH{\bfq}$.
This strategy is very general and can be virtually applied to any real-domain algorithm.
Importantly, it also ensures that the convergence properties of the quaternion-domain algorithms are directly inherited from their augmented real counterparts.

To illustrate the proposed methodology, we derive in the sequel the quaternion version of the popular alternating direction method of multipliers (ADMM), which we simply call quaternion ADMM (Q-ADMM).
This focus is motivated by the fact that its real-domain counterpart \cite{boyd_distributed_2010} can accomodate a large variety of constraints together while maintaining simple and efficient updates.
As such,  Q-ADMM appears as a versatile algorithm for quaternion-domain optimization.
Note that there have been several attempts to formulate ADMM for quaternion-domain optimization problems: they either rely on a real augmented formulation \cite{zou_quaternion_2016,zou_quaternion_2021} or are particularly designed for specific applications \cite{miao_quaternion-based_2020,miao_low-rank_2020}.
In contrast, this paper introduces a general Q-ADMM framework by leveraging the proposed quaternion convex optimization framework.

Now, consider the general quaternion optimization problem:
\begin{equation}
  \begin{split}
    \text{minimize } & f(\bfq) + g(\bfp)\\
    \text{subject to } & \bfA_1 \bfq + \bfA_2 \invol{\bfq}{\bmi} + \bfA_3\invol{\bfq}{\bmj} + \bfA_4 \invol{\bfq}{\bmk} \\
    &+ \bfB_1 \bfp + \bfB_2 \invol{\bfp}{\bmi} + \bfB_3\invol{\bfp}{\bmj} + \bfB_4 \invol{\bfp}{\bmk} = \bfc\:,
  \end{split}\label{eq:admmQuaternionGeneric}
\end{equation}
where $f$ and $g$ are real-valued convex functions of quaternion variables $\bfq \in \bbH^n$ and $\bfp \in \bbH^m$, respectively.
The two variables are linked through a widely affine constraint, defined by $\bfA_i \in \bbH^{p\times n}, \bfB_i \in \bbH^{p\times m}$ for $i=1, 2, 3, 4$ and $\bfc \in \bbH^p$.
Note that since $f$ and $g$ are supposed convex, problem \eqref{eq:admmQuaternionGeneric} is a widely affine equality constrained convex quaternion optimization problem.
Once again, widely affine relations in $\bfq$ and $\bfp$ variables are necessary to ensure that \eqref{eq:admmQuaternionGeneric} encodes the generality of convex quaternion equality constraints.

Quaternion ADMM aims at solving \eqref{eq:admmQuaternionGeneric} in its quaternion variables $\bfq$ and $\bfp$.
To develop this algorithm described in Section \ref{sub:quaternionADMM} below, we start by deriving the standard ADMM updates for the real-augmented problem associated with \eqref{eq:admmQuaternionGeneric}.
Then, by considering equivalencies, one obtains the corresponding algorithms in the quaternion augmented representation $\calH$ and eventually for $\bbH^n$.
For completeness, note that a special instance of Q-ADMM has been proposed recently in \cite{qadmm2021}, where $\bfA_i = \bfB_i = \mathbf{0}$ for $i=1,2,3$ in \eqref{eq:admmQuaternionGeneric}.

\subsection{ADMM in real augmented domain}
The original real-domain ADMM algorithm \cite{boyd_distributed_2010} can be directly applied to the real-augmented optimization problem equivalent to \eqref{eq:admmQuaternionGeneric}, which reads
\begin{equation}
  \begin{split}
    \text{minimize } & f(\aR{\bfq}) + g(\aR{\bfp})\\
    \text{subject to } & \aR{\bfA}\aR{\bfq} + \aR{\bfB}\aR{\bfp} = \aR{\bfc}
  \end{split}\label{eq:admmRealAugmentedGeneric}
\end{equation}
where $\aR{\bfq} \in \calR_n$ and $\aR{\bfp} \in \calR_m$ are augmented real variables,  and where $\aR{\bfA} \in \bbR^{4p\times 4n}, \aR{\bfB} \in \bbR^{4p\times 4m}$ and $\aR{\bfc} \in \bbR^{4p}$ encode a general affine relation between augmented real variables $\aR{\bfq}$ and $\aR{\bfp}$.
First, define the augmented Lagrangian for $\rho \geq 0$ and Lagrange multiplier $\aR{\boldsymbol{\lambda}} \in \calR_p$:
\begin{equation}
  \begin{split}
    &L_\rho(\aR{\bfq}, \aR{\bfp}, \aR{\boldsymbol{\lambda}})\\
    &\triangleq f(\aR{\bfq}) + g(\aR{\bfp})+ \aR{\boldsymbol{\lambda}}^\transp \left( \aR{\bfA}\aR{\bfq}     + \aR{\bfB}\aR{\bfp} - \aR{\bfc}\right)  \\
    &+  \frac{\rho}{2} \left\Vert \aR{\bfA}\aR{\bfq} + \aR{\bfB}\aR{\bfp} - \aR{\bfc}\right\Vert_{2}^2\:.
  \end{split}
\end{equation}
ADMM updates then consist of the iterations
\begin{align}
  \aR{\bfq}^{(k+1)}                 & = \argmin_{\aR{\bfq}} L_\rho(\aR{\bfq}, \aR{\bfp}^{(k)}, \aR{\boldsymbol{\lambda}}^{(k)})                                    \\
  \aR{\bfp}^{(k+1)}                 & = \argmin_{\aR{\bfp}} L_\rho(\aR{\bfq}^{(k+1)}, \aR{\bfp}, \aR{\boldsymbol{\lambda}}^{(k)})                                  \\
  \aR{\boldsymbol{\lambda}}^{(k+1)} & = \aR{\boldsymbol{\lambda}}^{(k)} + \rho\left(\aR{\bfA}\aR{\bfq}^{(k+1)}     + \aR{\bfB}\aR{\bfp}^{(k+1)} - \aR{\bfc}\right)
\end{align}
Defining the \emph{scaled dual variable} as $\aR{\bfu} = (1/\rho) \aR{\boldsymbol{\lambda}}$ one obtains the \emph{scaled form} for ADMM
\begin{align}
  \aR{\bfq}^{(k+1)} & = \argmin_{\aR{\bfq}}\left\lbrace f(\aR{\bfq}) \phantom{\frac{\rho}{2}}\right. \nonumber                                                          \\
                    & \left. + \frac{\rho}{2}\left\Vert \aR{\bfA}\aR{\bfq}     + \aR{\bfB}\aR{\bfp}^{(k)} - \aR{\bfc} + \aR{\bfu}^{(k)}\right\Vert_{2}^2\right\rbrace   \\
  \aR{\bfp}^{(k+1)} & = \argmin_{\aR{\bfp}}\left\lbrace g(\aR{\bfp}) \phantom{\frac{\rho}{2}}\right. \nonumber                                                          \\
                    & \left. + \frac{\rho}{2}\left\Vert \aR{\bfA}\aR{\bfq}^{(k+1)}     + \aR{\bfB}\aR{\bfp} - \aR{\bfc} + \aR{\bfu}^{(k)}\right\Vert_{2}^2\right\rbrace \\
  \aR{\bfu}^{(k+1)} & = \aR{\bfu}^{(k)} + \aR{\bfA}\aR{\bfq}^{(k+1)}     + \aR{\bfB}\aR{\bfp}^{(k+1)} - \aR{\bfc}\:.
\end{align}
Equipped with these expressions, the goal is now to find augmented quaternion domain and quaternion domain counterparts for the expressions above.

\subsection{ADMM in quaternion augmented domain}
The linear relationship between $\calR$ and $\calH$ permits a simple derivation of ADMM in quaternion augmented variables.
First, note that by Proposition \ref{prop:conservationInnerProduct} one has
\begin{equation}
  \left\Vert \aR{\bfA}\aR{\bfq} + \aR{\bfB}\aR{\bfp} - \aR{\bfc}\right\Vert_{2}^2 = \left\Vert \aH{\bfA}\aH{\bfq} + \aH{\bfB}\aH{\bfp} - \aH{\bfc}\right\Vert_{2}^2
\end{equation}
so that together with the expression of the Lagrangian in augmented quaternion variables \eqref{eq:LagrangianaH} the augmented Lagrangian in augmented quaternion variables reads:
\begin{equation}
  \begin{split}
    L_\rho(\aH{\bfq}, \aH{\bfp}, \aH{\boldsymbol{\lambda}})  &= f(\aH{\bfq}) + g(\aH{\bfp})\\
    & + \frac{1}{4}\aH{\boldsymbol{\lambda}}^\herm \left( \aH{\bfA}\aH{\bfq}     + \aH{\bfB}\aH{\bfp} - \aH{\bfc}\right)  \\
    &+  \frac{\rho}{2} \left\Vert \aH{\bfA}\aH{\bfq} + \aH{\bfB}\aH{\bfp} - \aH{\bfc}\right\Vert_{2}^2\:.
  \end{split}
\end{equation}
Then, since optimal points are the same regardless the representation (see Proposition \ref{prop:stationaryPoints}) the first two iterates of ADMM are identical.
The dual ascent step is obtained once again by exploiting the linear relationship between $\calR$ and $\calH$.
The augmented Q-ADMM iterates thus read
\begin{align}
  \aH{\bfq}^{(k+1)}                 & = \argmin_{\aH{\bfq}} L_\rho(\aH{\bfq}, \aH{\bfp}^{(k)}, \aH{\boldsymbol{\lambda}}^{(k)})                                   \\
  \aH{\bfp}^{(k+1)}                 & = \argmin_{\aH{\bfp}} L_\rho(\aH{\bfq}^{(k+1)}, \aH{\bfp}, \aH{\boldsymbol{\lambda}}^{(k)})                                 \\
  \aH{\boldsymbol{\lambda}}^{(k+1)} & = \aH{\boldsymbol{\lambda}}^{(k)} + \rho\left(\aH{\bfA}\aH{\bfq}^{(k+1)} + \aH{\bfB}\aH{\bfp}^{(k+1)} - \aH{\bfc}\right)\:.
\end{align}
Similar arguments and computations give the scaled form iterates, which are omitted for space considerations.

\subsection{ADMM in quaternion domain}
\label{sub:quaternionADMM}
ADMM iterates in quaternion domain can now be obtained directly from expressions above.
For the sake of notation brevity, let us introduce the quaternion residual $\bfr(\bfq, \bfp)$ as
\begin{equation}
  \begin{split}
    \bfr(\bfp, \bfq) &=    \bfA_1 \bfq + \bfA_2 \invol{\bfq}{\bmi} + \bfA_3\invol{\bfq}{\bmj} + \bfA_4 \invol{\bfq}{\bmk} \\
    &+ \bfB_1 \bfp + \bfB_2 \invol{\bfp}{\bmi} + \bfB_3\invol{\bfp}{\bmj} + \bfB_4 \invol{\bfp}{\bmk} - \bfc \label{eq:quaternionResidualDefinition}
  \end{split}
\end{equation}
and recall that by Proposition \ref{prop:conservationInnerProduct} one has $\Vert  \bfr(\bfp, \bfq) \Vert_{2}^2 = \Vert  \aH{\bfr}(\aH{\bfp}, \aH{\bfq}) \Vert^2_{2}$.
The augmented Lagrangian in quaternion variables then reads
\begin{equation}
  \begin{split}
    L_\rho(\bfq, \bfp, \boldsymbol{\lambda}) & = f(\bfq) + g(\bfp)\\ &+\real\left(\boldsymbol{\lambda}^\herm \bfr(\bfp, \bfq) \right) +  \frac{\rho}{2} \left\Vert \bfr(\bfp, \bfq)\right\Vert_{2}^2
  \end{split}
\end{equation}
The ADMM iterates in the quaternion domain are very similar to their corresponding real augmented counterparts
\begin{align}
  \bfq^{(k+1)}                 & = \argmin_{\bfq} L_\rho(\bfq, \bfp^{(k)}, \boldsymbol{\lambda}^{(k)})   \\
  \bfp^{(k+1)}                 & = \argmin_{\bfp} L_\rho(\bfq^{(k+1)}, \bfp, \boldsymbol{\lambda}^{(k)}) \\
  \boldsymbol{\lambda}^{(k+1)} & = \boldsymbol{\lambda}^{(k)} +\rho \bfr(\bfq^{(k+1)}, \bfp^{(k+1)})
\end{align}
and can be expressed in scaled form as
\begin{align}
  \bfq^{(k+1)} & = \argmin_{\bfq}\left\lbrace f(\bfq) + \frac{\rho}{2}\left\Vert \bfr(\bfq, \bfp^{(k)}) +  \bfu^{(k)}\right\Vert_{2}^2\right\rbrace \label{eq:dualScaledQuaternionADMMq} \\
  \bfp^{(k+1)} & = \argmin_{\bfp}\left\lbrace g(\bfp) + \frac{\rho}{2}\left\Vert \bfr(\bfq^{(k+1)}, \bfp) +  \bfu^{(k)}\right\Vert_{2}^2\right\rbrace                                    \\
  \bfu^{(k+1)} & = \bfu^{(k)} + \bfr(\bfq^{(k+1)}, \bfp^{(k+1)})\:.\label{eq:dualScaledQuaternionADMMqend}
\end{align}

\subsection{Convergence of Q-ADMM}
Q-ADMM inherits its convergence properties from the convergence results of the associated augmented real ADMM algorithm \cite{boyd_distributed_2010}, which are adapted to the quaternion case below for completeness.
More precisely, let us make the standard assumptions that the extended-real-valued functions $f:\bbH^n \rightarrow \bbR$ and $g: \bbH^m \rightarrow \bbR$ are closed\footnote{Alike in standard real optimization \cite[A.3.3]{boyd_convex_2004}, a function $h:\bbH^n \rightarrow \bbR$ is said to be closed if, for each $\alpha \in \bbR$, the sublevel set $\lbrace \bfq \in \dom h \vert h(\bfq)\leq \alpha\rbrace$ is closed. Recall that a set $\calS$ is closed if its complement $\bbH^n \setminus \calS$ is open. A quaternion set $\calS \subseteq \bbH^n$ is open iff $\forall \bfq \in \calS$, $\exists \epsilon > 0$ such that $\lbrace \bfp \mid \Vert \bfq - \bfp \Vert_2 \leq \epsilon \rbrace\subseteq \calS$. }, proper and convex.
We also assume that there exist at least one $(\tilde{\bfq}, \tilde{\bfp}, \tilde{\boldsymbol{\lambda}})$ such that $L_0(\tilde{\bfq}, \tilde{\bfp}, {\boldsymbol{\lambda}}) \leq L_0(\tilde{\bfq}, \tilde{\bfp}, \tilde{\boldsymbol{\lambda}}) \leq L_0(\bfq, \bfp, \tilde{\boldsymbol{\lambda}})$ for all $\bfq, \bfp, \boldsymbol{\lambda}$.
Under these two assumptions, the Q-ADMM iterates satisfy:
\begin{itemize}
  \item Convergence of the quaternion residual \eqref{eq:quaternionResidualDefinition}: $\bfr(\bfq^{(k)}, \bfp^{(k)}) \rightarrow 0$ as $k\rightarrow \infty$, ;
  \item Objective convergence: $f(\bfq^{(k)}) + g(\bfp^{(k)}) \rightarrow \tilde{v}$, where $\tilde{v}$ is the optimal value of \eqref{eq:admmQuaternionGeneric};
  \item Dual variable convergence: $\boldsymbol{\lambda}^{(k)} \rightarrow \tilde{\boldsymbol{\lambda}}$ as $k\rightarrow \infty$\:.
\end{itemize}
Necessary and sufficient conditions for primal-dual optimality of the triplet $(\tilde{\bfq}, \tilde{\bfp}, \tilde{\boldsymbol{\lambda}})$ can be directly obtained  for the Q-ADMM problem \eqref{eq:admmQuaternionGeneric} using the KKT conditions \eqref{eq:KKT1}--\eqref{eq:KKTend}; they are omitted here for space considerations.

\subsection{Special case: proximal operator form}
\label{sub:specialCaseAdmm}
Consider the special case where the affine constraint in \eqref{eq:admmQuaternionGeneric} is simply given by $\bfq -\bfp = 0$, that is, $m=n$, $\bfA_1 = \bfB_1 = \bfI_n$ and $\bfA_i=\bfB_i= \mathbf{0}_n$ for $i=2, 3, 4$.
Q-ADMM iterations then become
\begin{align}
  \bfq^{(k+1)} & = \argmin_{\bfq}\left\lbrace f(\bfq) + \frac{\rho}{2}\left\Vert \bfq  -\bfp^{(k)} + \bfu^{(k)} \right\Vert_{2}^2\right\rbrace\label{eq:quaternionADMMqProx} \\
  \bfp^{(k+1)} & = \argmin_{\bfp}\left\lbrace g(\bfp) + \frac{\rho}{2}\left\Vert \bfq^{(k+1)}  -\bfp + \bfu^{(k)} \right\Vert_{2}^2\right\rbrace                             \\
  \bfu^{(k+1)} & = \bfu^{(k)} + \bfq^{(k+1)}-\bfp^{(k+1)}\:.
\end{align}
Focusing on the $\bfq$-update for simplicity, \eqref{eq:quaternionADMMqProx} can be rewritten as
\begin{equation}
  \bfq^{(k+1)} \triangleq \prox_{f/\rho} (\bfp^{(k)} - \bfu^{(k)}) \label{eq:defQuaternionProx}
\end{equation}
where $\prox_{f/ \rho}$ denotes the \emph{quaternion proximal} operator of $f$ with penalty $\rho$, first introduced in \cite{chan_complex_2016}.
Importantly, when $f$ is the indicator function on a closed convex set $\calC \subset \bbH^n$, the $\bfq$-update becomes
\begin{equation}
  \bfq^{(k+1)} = \argmin_{\bfq \in \calC} \left\Vert  \bfq  -\bfp^{(k)} + \bfu^{(k)} \right\Vert_{2}^2 \triangleq \Pi_{\calC}\left(\bfp^{(k)} - \bfu^{(k)}\right)
\end{equation}
where $\Pi_{\calC}$ denotes the projection onto $\calC$ in the quaternion Euclidean norm.

\section{The framework in practice}
\label{sec:examples}
This last section illustrates the relevance of the proposed framework by considering two general examples of constrained convex optimization problems in quaternion variables.
Both problems can be solved efficiently using the Q-ADMM algorithm introduced in Section \ref{sec:quaternionADMM}.
Since Q-ADMM share the same convergence and numerical properties with its real-augmented domain counterpart, we only focus hereafter on the many insights enabled by the quaternion framework.
Note that all quaternion-domain algorithms presented in this section can safely be implemented in \textsc{Matlab}\textregistered\:using \texttt{QTFM}\footnote{\emph{Quaternion toolbox for Matlab}, available online at \url{https://sourceforge.net/projects/qtfm/.}}.

\subsection{Constrained widely linear least squares}
\label{sub:constrainedWLLS}
Quaternion widely linear models have attracted a lot of interest in recent years, see for instance \cite{took2010quaternion,jahanchahi_class_2013,jahanchahi_class_2014,navarro-moreno_quaternion_2014}, among others.
Importantly, resulting (unconstrained) widely linear quaternion estimators yield optimal mean square error estimation when considering improper (\ie non-circular) 3D or 4D data.
A natural generalization of this approach is to require the solution to obey to some convex constraints, leading to  the following \emph{constrained widely linear least squares problem}
\begin{equation}
  \begin{split}
    \text{minimize } & \frac{1}{2}\Vert \bfy - \bfP_1 \bfq  - \bfP_2 \invol{\bfq}{\bmi} -\bfP_3\invol{\bfq}{\bmj} - \bfP_4 \invol{\bfq}{\bmk}\Vert_2^2\\
    \text{subject to } & \bfq \in \calC\:,
  \end{split}\label{prob:constrainedWLLS}
\end{equation}
where $\calC$ is a closed convex subset of $\bbH^n$.
An important special case is when $\calC$ is a convex cone, in particular to encode \emph{quaternion non-negative} constraints.
For instance, in \cite{mizoguchi_hypercomplex_2019} the authors enforce each component of the vector to have non-negative real and imaginary parts; in \cite{flamant_quaternion_2020}, each component of the vector must obey to $q_a \geq 0$ and $q_a^2 \geq q_b^2 + q_c^2 + q_d^2$, which can be interpreted as the non-negative definiteness of a specific 2-by-2 complex matrix.

The general constrained widely linear least square problem \eqref{prob:constrainedWLLS} can be solved efficiently using Q-ADMM.
Following a standard procedure \cite{boyd_distributed_2010}, the problem \eqref{prob:constrainedWLLS} can be cast as
\begin{align}
  \text{minimize }   & \frac{1}{2}\Vert \bfy - \bfP_1 \bfq  - \bfP_2 \invol{\bfq}{\bmi} -\bfP_3\invol{\bfq}{\bmj} - \bfP_4 \invol{\bfq}{\bmk}\Vert_2^2 + \iota_{\calC}(\bfp)\nonumber \\
  \text{subject to } & \bfq -\bfp = 0
\end{align}
where $\iota_{\calC}$ denotes the indicator function of $\calC$ such that $\iota_{\calC}(\bfp) = 0$ for $\bfp \in \calC$ and $\iota(\bfp) = +\infty$ otherwise.
Q-ADMM iterations can be directly applied:
\begin{align}
  \bfq^{(k+1)} & =  \argmin_{\bfq} \left\lbrace \Vert \bfy - \bfP_1 \bfq  - \bfP_2 \invol{\bfq}{\bmi} -\bfP_3\invol{\bfq}{\bmj} - \bfP_4 \invol{\bfq}{\bmk}\Vert_2^2 \right. \nonumber \\
               & \left. + \rho\Vert \bfq - \bfp^{(k)} + \bfu^{(k)}\Vert_2^2\right\rbrace \label{eq:subproblemQWLLS}                                                                    \\
  \bfp^{(k+1)} & =  \Pi_{\calC}\left(\bfq^{(k+1)} + \bfu^{(k)}\right)                                                                                                                  \\
  \bfu^{(k+1)} & = \bfu^{(k)} + \bfq^{(k+1)}-\bfp^{(k+1)}\:.
\end{align}
In a nutshell, Q-ADMM consists in iteratively solving a general unconstrained quaternion least squares problem followed by projection onto the constraint set $\calC$ and dual ascent.
Such formulation is thus particulary useful when projection onto $\calC$ can be carried out explicitly, as in the case of non-negative constraints mentioned above \cite{flamant_quaternion_2020,mizoguchi_hypercomplex_2019}.

\noindent\textbf{Solving the $\bfq$-variable subproblem \eqref{eq:subproblemQWLLS}}
Since the function to be minimized involves $\bfq$ and its three canonical involutions $\invol{\bfq}{\bmi}, \invol{\bfq}{\bmj}, \invol{\bfq}{\bmk}$, finding a solution is not as straightforward as in the linear case $(\bfP_2 = \bfP_3= \bfP_4 = \mathbf{0}$).
Two strategies are essentially possible: \emph{(i)} obtain an explicit expression for $\bfq^{(k+1)}$ using the augmented real $\calR$ or the augmented quaternion representation $\calH$; \emph{(ii)} solve iteratively for \eqref{eq:subproblemQWLLS} \emph{e.g.} using quaternion gradient descent in $\bbH^n$.
We describe below these two approaches in detail.

\paragraph{Explicit solution via augmented representations} The $\bfq$-update \eqref{eq:subproblemQWLLS} can be rewritten in $\calR$ as the $\aR{\bfq}$-update
\begin{equation}
  \aR{\bfq}^{(k+1)} = \argmin_{\aR{\bfq}} \left\lbrace \left\Vert \aR{\bfP}\aR{\bfq} - \aR{\bfy}\right\Vert_{2}^2 + \rho \left\Vert \aR{\bfq} - \aR{\bfv}\right\Vert_{2}^2 \right\rbrace
\end{equation}
where $\aR{\bfP} \in \bbR^{4m\times 4n}, \aR{\bfy} \in \bbR^{4m}$ and $\aR{\bfv} = \aR{\bfp}^{(k)} - \aR{\bfu}^{(k)} \in \bbH^{4n}$ is constant.
Being a standard real optimization problem, the optimality condition reads:
\begin{equation}
  \aR{\bfP}^\transp\left(\aR{\bfP}\aR{\bfq} - \aR{\bfy}\right) + \rho\left(\aR{\bfq} - \aR{\bfv}\right) = 0
\end{equation}
so that one gets easily the standard explicit solution
\begin{equation}
  \aR{\bfq}^{(k+1)} = \left(\aR{\bfP}^\transp\aR{\bfP}+ \rho \bfI_{4n}\right)^{-1}\left(\aR{\bfP}^\transp \aR{\bfy} + \rho\aR{\bfv}\right)\:.
\end{equation}
Exploiting the relation between $\calR$ and $\calH$ augmented representations, we also get
\begin{equation}
  \aH{\bfq}^{(k+1)} = \left(\aH{\bfP}^\herm\aH{\bfP}+ \rho\bfI_{4n}\right)^{-1}\left(\aH{\bfP}^\herm \aH{\bfy} + \rho\aH{\bfv}\right)
\end{equation}
so that the explicit $\bfq$-update reads
\begin{equation}
  \bfq^{(k+1)} = \bfS\left(\aH{\bfP}^\herm\aH{\bfP}+ \rho\bfI_{4n}\right)^{-1}\left(\aH{\bfP}^\herm \aH{\bfy} + \rho\aH{\bfv}\right)\label{eq:WLSSexplicitq}
\end{equation}
where $\bfS = \begin{bmatrix}
    \bfI_n & \mathbf{0}_{n} & \mathbf{0}_{n} & \mathbf{0}_{n}
  \end{bmatrix}\in \bbR^{n\times 4n}$ selects the first $n$ entries of $\aH{\bfq}^{(k+1)}$.
Notably, when the cost function is linear quadratic, \emph{i.e.} when $\bfP_i = 0$ for $i=2, 3, 4$, Eq. \eqref{eq:WLSSexplicitq} becomes
\begin{equation}
  \bfq^{(k+1)} = \left(\bfP_1^\herm\bfP_1+ \rho\bfI_{n}\right)^{-1}\left(\bfP_1^\herm \bfy + \rho\bfv\right)\:.
\end{equation}
Unfortunately for the general (widely linear) case, such expression of $\bfq^{(k+1)}$ in terms of quaternion-domain matrices $\bfP_i$, $i=1, 2,3,4$ is not possible.
One has to turn back to \eqref{eq:WLSSexplicitq}, which requires inversion of augmented matrices of size $4n$, thus limiting its application for large scale applications when $n$ is large.

\paragraph{Iterative scheme via quaternion gradient descent}
Another possibility is to use an iterative scheme to approximately solve \eqref{eq:subproblemQWLLS}.
We choose here quaternion gradient descent \cite{mandic_quaternion_2011,xu2016optimization}, which takes the form
\begin{equation}
  \begin{split}
    \bfq^{(\ell+1)} = \bfq^{(\ell)} - \eta_\ell \conjgradQ h(\bfq^{(\ell)})\:,\label{eq:gradienDescentSubproblem}
  \end{split}
\end{equation}
where $\eta_\ell$ is an iteration dependent step-size and where $h(\bfq)$ is defined by
\begin{equation}
  \begin{split}
    h(\bfq) &= \left\Vert \bfP_1 \bfq+ \bfP_2 \invol{\bfq}{\bmi} + \bfP_3\invol{\bfq}{\bmj} + \bfP_4 \invol{\bfq}{\bmk}- \bfy \right\Vert_{2}^2  \\
    &+ \rho\left\Vert \bfq - \bfp^{(k)} + \bfu^{(k)} \right\Vert_{2}^2\:.
  \end{split}
\end{equation}
These iterations provide an approximate solution to the $\bfq$-optimization problem \eqref{eq:subproblemQWLLS} such that $\bfq^{(k+1)} = \bfq^{(\ell_0)}$, where $\ell_0$ is defined by an appropriate stopping criterion controlling the accuracy of the solution.
Explicit iterations can be obtained by using results from Appendix \ref{app:WLgradient}:
\begin{align}
  \bfq^{(\ell+1)} & = \bfq^{(\ell)}  - \frac{\eta_\ell}{2} \left\lbrace\bfP_1^\herm \bfr^{(\ell)}_{\bfP} + \invol{\left(\bfP_2^\herm \bfr^{(\ell)}_{\bfP}\right)}{\bmi} + \invol{\left(\bfP_3^\herm \bfr^{(\ell)}_{\bfP}\right)}{\bmj} \right. \nonumber \\
                  & \left.  + \invol{\left(\bfP_4^\herm \bfr^{(\ell)}_{\bfP}\right)}{\bmk}\right\rbrace - \frac{\rho\eta_\ell}{2}\left[\bfq^{(\ell)} - \bfp^{(k)} + \bfu^{(k)}\right]
\end{align}
where $\bfr^{(\ell)}_{\bfP} := \bfr_{\bfP}(\bfq^{(\ell)})$ such that $\bfr_{\bfP}(\bfq) = \bfP_1 \bfq + \bfP_2 \invol{\bfq}{\bmi} + \bfP_3\invol{\bfq}{\bmj} + \bfP_4 \invol{\bfq}{\bmk}- \bfy$.
Compared with the explicit solution \eqref{eq:WLSSexplicitq} of the subproblem \eqref{eq:subproblemQWLLS}, the approximate iterative solution does not require quaternion matrix inversion.
This is particularly interesting with ADMM, since overall convergence of the algorithm can still be guaranteed even when minimization of subproblems is carried out approximately \cite{boyd_distributed_2010}.
As a consequence, performing a few (cheap) quaternion gradient steps \eqref{eq:gradienDescentSubproblem} at each iteration $k$ will ensure convergence of Q-ADMM to a stationary point of the cost function.

\subsection{3D basis pursuit denoising}
\label{sub:3DLasso}

A major application of quaternion algebra lies in its ability to represent 3D data, including color images \cite{le2003quaternion,chen2015color}, wind data \cite{took2011quaternion}, seismics \cite{miron2006quaternion}, among others.
A dataset comprising $N$ 3D vectors is coded as a pure quaternion vector of dimension $N$: for instance, a color image patch is described as the pure quaternion vector $\bfq = \bmi \bfr + \bmj\bfg + \bmk \bfb$, where $\bfr, \bfg, \bfb$ are real vectors denoting the red, green and blue components of the image.
With this representation, we consider the general pure quaternion (or 3D) basis pursuit denoising problem, first formulated in the color image processing literature \cite{xu_vector_2015,barthelemy_color_2015,zou_quaternion_2016}.
3D data measurements $\bfy$ are supposed to follow the linear quaternion model $\bfy = \bfD\bfq + \bfn$, where $\bfD \in \bbH^{m\times n}$ is the dictionary ($m < n$), $\bfq \in \bbH^{n}$ is the vector of sparse coefficients and $\bfn$ is the noise.
While the relevance of the quaternion model has been established by many authors, its interpretability is conditioned by the reconstructed 3D signal $\bfD\bfq$ being a pure quaternion vector, that is $\real(\bfD\bfq) = 0$.
Unfortunately, this is not the case in general, when $\bfD$ and $\bfq$ are quaternion-valued and must be enforced within the algorithm in order to obtain interpretable solutions.
Currently \cite{barthelemy_color_2015}, the constraint $\real(\bfD\bfq) = 0$ is generally imposed by simply nulling the real part of the product $\bfD\bfq$ -- which does not preserve convergence properties.
The proposed algorithm hereafter solves this issue of existing algorithms by leveraging the Q-ADMM framework.

The resulting 3D basis pursuit denoising can be formulated as follows
\begin{equation}
  \begin{split}
    \text{minimize } & \frac{1}{2}\Vert \bfy - \bfD \bfq \Vert_2^2 + \beta \Vert \bfq \Vert_1 \\
    \text{subject to } & \real (\bfD\bfq) = 0\:,\label{prob:3DLasso}
  \end{split}
\end{equation}
where $\beta > 0$ is a parameter that controls the amount of sparsity.
The quaternion $\ell_1$-norm promotes sparsity and is defined by
\begin{equation}
  \Vert \bfq \Vert_1 \triangleq \sum_{i=1}^n \vert q_i\vert = \sum_{i=1}^n \sqrt{q_{ai}^2 + q_{bi}^2 + q_{ci}^2 + q_{di}^2}\:.
\end{equation}
As noted in \cite{xu_vector_2015}, the quaternion $\ell_1$-norm is equivalent to the real $\ell_{2,1}$-norm in $\bbR^{n\times 4}$, meaning that quaternion Lasso can be seen as an instance of real-valued group Lasso where groups are composed of the real and three imaginary parts of a quaternion.
The constraint $\real (\bfD\bfq) = 0$ is widely affine since
\begin{equation}
  \real (\bfD\bfq) = 0 \Leftrightarrow \bfD\bfq + \invol{(\bfD\bfq)}{\bmi} + \invol{(\bfD\bfq)}{\bmj} + \invol{(\bfD\bfq)}{\bmk}= 0
\end{equation}
This ensures that $\eqref{prob:3DLasso}$ defines a quaternion convex optimization problem, which can be rewritten in Q-ADMM form as
\begin{equation}
  \begin{split}
    \text{minimize } & \frac{1}{2}\Vert \bfy - \bfD\bfq \Vert^2_{2} + \beta \Vert \bfp \Vert_1\\
    \text{subject to } & \bfq -\bfp = 0 \text{ and } \real(\bfD\bfq) = 0 \:.
  \end{split}
\end{equation}
Q-ADMM iterations then read
\begin{align}
  \bfq^{(k+1)} & = \argmin_{\substack{\bfq                                                                \\ \real(\bfD\bfq) = 0}} \frac{1}{2}\Vert \bfy - \bfD\bfq \Vert^2_{2} + \frac{\rho}{2}\Vert \bfq - \bfp^{(k)} + \bfu^{(k)}\Vert_2^2\label{eq:quaternionSparseADMM1}\\
  \bfp^{(k+1)} & = \prox_{\frac{\beta}{\rho}\Vert \cdot \Vert_1}\left({\bfq}^{(k+1)} +  \bfu^{(k)}\right) \\
  \bfu^{(k+1)} & = \bfu^{(k)} + {\bfq}^{(k+1)}-\bfp^{(k+1)}\:.\label{eq:quaternionSparseADMM2}
\end{align}
The $\bfq$-update is a widely affine constrained least squares problem, which can be tackled by solving the KKT conditions \eqref{eq:KKT1}--\eqref{eq:KKTend}, see details below.
The $\bfp$-update involves the computation of the proximal operator of the quaternion $\ell_1$-norm, whose expression can be found in the existing literature \cite{chan_complex_2016}:
\begin{equation}
  \prox_{\lambda \Vert \cdot \Vert_1}(\bfq) = \left[\max\left(0, 1- \frac{\lambda}{\vert q_i\vert}\right)q_i\right]_{i=1\ldots n}\triangleq \calS_{\lambda}(\bfq)\:,
\end{equation}
where $\calS_{\lambda}(\cdot)$ is the quaternion soft-thresholding operator.
As a result, the $\bfp$-update becomes
\begin{equation}
  \bfp^{(k+1)} = \calS_{\frac{\beta}{\rho}}\left({\bfq}^{(k+1)} +  \bfu^{(k)}\right)\:.
\end{equation}
\noindent\textbf{Solving the $\bfq$-variable subproblem \eqref{eq:quaternionSparseADMM1}} KKT conditions  \eqref{eq:KKT1}--\eqref{eq:KKTend} for the widely affine constrained least squares problem \eqref{eq:quaternionSparseADMM1} give necessary and sufficient conditions for $(\tilde{\bfq}, \tilde{\boldsymbol{\lambda}})$ to be primal and dual optimal:
\begin{align}
   & \bfD \tilde{\bfq} + \invol{(\bfD\tilde{\bfq})}{\bmi} + \invol{(\bfD\tilde{\bfq})}{\bmj}+ \invol{(\bfD\tilde{\bfq})}{\bmk} =0 \\
   & \frac{1}{4}\bfD^\herm \left( \bfD\tilde{\bfq} - \bfy\right) +
  \frac{\rho}{4}\left(\tilde{\bfq} - \bfv\right)\nonumber
  \\ &+ \frac{1}{4}\left[\bfD^\herm\tilde{\boldsymbol{\lambda}} + \bfD^\herm\invol{\tilde{\boldsymbol{\lambda}}}{\bmi} + \bfD^\herm\invol{\tilde{\boldsymbol{\lambda}}}{\bmj} + \bfD^\herm\invol{\tilde{\boldsymbol{\lambda}}}{\bmk}\right] = 0
\end{align}
where $\bfv = \bfp^{(k)} - \bfu^{(k)}$.
Solving for $\tilde{\bfq}$ gives
\begin{align}
  \tilde{\bfq} & = \left(\bfD^\herm \bfD + \bfI_n\right)^{-1}\left(\rho\bfv + \bfD^\herm \bfy -4\bfD^\herm \real \tilde{\boldsymbol{\lambda}}\right)\:.\label{eq:tildeqsparseADMM}
\end{align}
Let $\tilde{\bfq}_{\text{un}} = \left(\bfD^\herm \bfD + \bfI_n\right)^{-1}\left(\rho\bfv + \bfD^\herm \bfy\right)$ denote the unconstrained solution to the least square problem \eqref{eq:quaternionSparseADMM1}.
Plugging \eqref{eq:tildeqsparseADMM} into the constraint $\real(\bfD\tilde{\bfq}) = 0$ gives the value of the real part\footnote{Note that the imaginary part of $\tilde{\boldsymbol{\lambda}}$ can be arbitrary, since the constraint $\real(\bfD\tilde{\bfq}) = 0$ is here between real-valued expressions.} of the Lagrange multiplier $\tilde{\boldsymbol{\lambda}}$:
\begin{equation}
  \real \tilde{\boldsymbol{\lambda}} = \frac{1}{4}\left[\real\left(\bfD(\bfD^\herm\bfD + \bfI_n)^{-1}\bfD^\herm\right)\right]^{-1} \real\left(\bfD\tilde{\bfq}_{\text{un}}\right)\:.\label{eq:realLambdasparseADMM}
\end{equation}
Note that for large $N$, the computational burden of \eqref{eq:realLambdasparseADMM} may be important due to the double matrix inverse operation.
To summarize, the subproblem \eqref{eq:quaternionSparseADMM1} is solved using KKT conditions as follows:
\begin{itemize}
  \item compute the unconstrained least square solution $\tilde{\bfq}_{\text{un}}$;
  \item compute $\real\tilde{\boldsymbol{\lambda}}$ using \eqref{eq:realLambdasparseADMM};
  \item obtain the explicit solution $\tilde{\bfq}$ to \eqref{eq:quaternionSparseADMM1} using \eqref{eq:tildeqsparseADMM}.
\end{itemize}
It is easily verified that if the unconstrained least square solutions satisfies the constraint, then $\real \tilde{\boldsymbol{\lambda}} = 0$ and the solution to \eqref{eq:quaternionSparseADMM1} is $\tilde{\bfq} = \tilde{\bfq}_{\text{un}}$.

\begin{figure*}[ht]
  \centering
  \includegraphics[width=\textwidth]{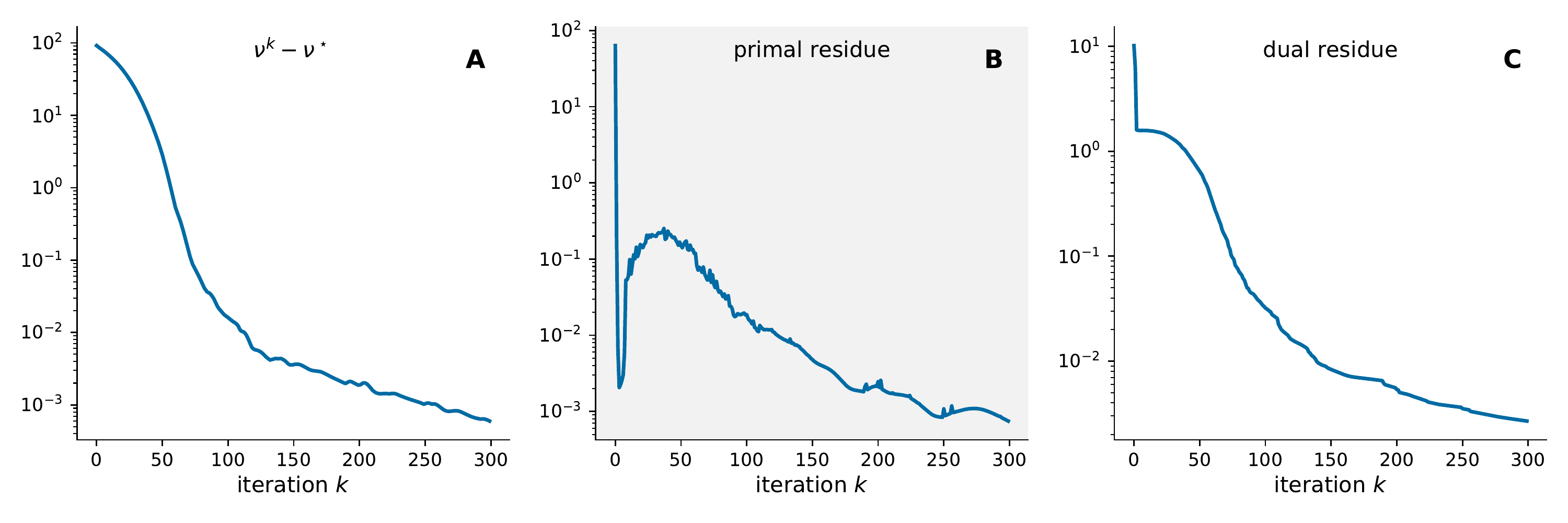}
  \caption{Numerical validation of the proposed quaternion ADMM algorithm for 3D basis pursuit denoising. Evolution of the suboptimality cost (\textbf{\textsf{A}}), norm of the primal residue (\textbf{\textsf{B}}) and norm of dual residue (\textbf{\textsf{C}}) with respect to iteration number $k$.}
  \label{fig:admmPerf}
\end{figure*}

\noindent\textbf{Numerical validation} For completeness, we provide experimental results which demonstrate the effectiveness of the proposed quaternion ADMM algorithm to solve the 3D basis pursuit denoising problem \eqref{prob:3DLasso}.
We generated a dictionary $\bfD \in \bbH^{10\times 1000}$ with random i.i.d. entries sampled from a quaternion unit Gaussian distribution.
Columns of $\bfD$ were normalized to 1.
Then, we fixed a sparse (about $3$ \% nonzeros entries) vector $\bfq_0 \in \bbH^{1000}$ such that $\real(\bfD\bfq_0) = 0$ to ensure a meaningful physical interpretation of $\bfD\bfq_0$ as a vector of 3D objects.
We further generated a noisy observation vector $\bfy = \bfD\bfq_0 + \bfn$, with $\bfn \in \bbH^{10}$ a pure quaternion random vector with i.i.d. entries such that $n_m = \bmi n_{am} + \bmj n_{bm} + \bmk n_{cm}$, and $n_{am}, n_{bm}, n_{cm} \sim \calN(0, \sigma^2)$.
A value of $\sigma = 10^{-1}$ was chosen in all numerical experiments.
We found empirically that the $\ell_1$-penalty parameter $\beta = 0.05$ and Q-ADMM parameter  $\rho = 1$ gave us the best results in terms of rate of convergence and correct identification of sparse components.
To assess the behaviour of the proposed Q-ADMM, we closely followed the approach proposed in \cite[Section 3.3]{boyd_distributed_2010}.

Figure \ref{fig:admmPerf}a) displays the evolution of the suboptimality cost $\nu^{(k)} - \nu^\star$ where $\nu^{(k)} = f(\bfq^{(k)}) + g(\bfp^{(k)}) = \frac{1}{2} \Vert \bfy - \bfD\bfq^{(k)}\Vert_2^2 + \beta \Vert \bfp^{(k)}\Vert_1$. The optimal value $\nu^\star$ of the cost function $f(\bfp) + g(\bfq)$ was obtained by running Q-ADMM for 1000 iterations.
As expected, $\nu^{(k)} - \nu^\star$ decreases to zero with the number of iterations $k$.
Figure \ref{fig:admmPerf}b) and \ref{fig:admmPerf}c) depicts, respectively, the evolution of the norm of the primal residue $\bfq^{(k+1)} - \bfp^{(k+1)}$ and the norm of dual residue $\rho \left(\bfp^{(k+1)} - \bfp^{(k)}\right)$ as iterations increase.
Similarly, both residues converge to zero as Q-ADMM proceeds.
These results show that Q-ADMM behaves similarly to  standard real ADMM \cite{boyd_distributed_2010}.
This is not surprising, since the proposed Q-ADMM algorithm (see Section \ref{sub:quaternionADMM}) was derived using systematic equivalencies with its real-augmented domain counterpart and therefore it enjoys the same convergence properties as standard ADMM.

\section{Conclusion}
\label{sec:conclusion}

In this paper, we provide a general and systematic quaternion-domain mathematical framework to solve constrained convex optimization problems formulated in quaternion variables.
This framework builds upon two key ingredients: \emph{i)} the recent theory of generalized \bbH\bbR-derivatives \cite{xu_enabling_2015} which allows the computation of derivatives with respect to quaternion variables, and \emph{ii)} the explicit correspondence between the original quaternion optimization problem and its real augmented counterpart, to provide sound and general mathematical guarantees.
Notably, it enables the formulation of fundamental quaternion-domain convex optimization tools, such as the quaternion Lagrangian function or KKT optimality conditions.
This methodology further reveals subtle properties of quaternion-domain optimization that would have been hindered otherwise, such as widely affine equality constraints.
For practical purposes, we derived quaternion ADMM (Q-ADMM) as a versatile framework for quaternion-domain optimization.
These results establish a systematic and general methodology to develop quaternion-domain algorithms for convex and nonconvex quaternion optimization problems.

We hope that the proposed framework will favor the development of full quaternion-domain methodologies to efficiently solve practical problems involving 3D and 4D data.
Indeed, when such a problem has a natural formulation in the quaternion domain, we expect the proposed quaternion optimization framework to fully take advantage of the many physical or geometric insights offered by the quaternion representation -- features that would be lost when resorting to an equivalent real formulation of the  problem at hand.
Together with the recent interest in high-performance hardware implementations of quaternion operations \cite{williams-young_efficacy_2019,joldecs2020algorithms}, the proposed framework pave the way to the generalization of quaternion-domain optimization procedures in a wide range of signal processing applications.

\appendices

\section{Quaternion gradient computations}
\label{app:WLgradient}
Consider the function $f: \bbH^n \rightarrow \bbR$ defined by $f(\bfq) = \frac{1}{2} \left\Vert \bfA_1 \bfq + \bfA_2 \invol{\bfq}{\bmi} + \bfA_3\invol{\bfq}{\bmj} + \bfA_4 \invol{\bfq}{\bmk}- \bfb \right\Vert_{2}^2$ where $\bfA_i \in \bbH^{p\times n}$ and $\bfb \in \bbH^p$ are arbitrary.
Using manipulations similar to those of Section \ref{sub:affine}, we get $f(\bfq) = f(\aR{\bfq}) = \frac{1}{2} \Vert \aR{\bfA}\aR{\bfq} - \aR{\bfb}\Vert_2^2$ with real augmented gradient
$$ \gradR f(\aR{\bfq}) = \aR{\bfA}^\transp\left(\aR{\bfA}\aR{\bfq} - \aR{\bfb}\right)$$
Knowing that $\aH{\bfA} = \frac{1}{4}\bfJ_p \aR{\bfA} \bfJ^\herm_n$ and $\aH{\bfq} = \bfJ_n \aR{\bfq}$, we get the augmented conjugated quaternion gradient from \eqref{eq:gradHconjtogradR} as
$$ \conjgradH f(\aH{\bfq}) = \frac{1}{4}\bfJ_n \gradR f(\bfq) = \frac{1}{4}\aH{\bfA}^\herm\left(\aH{\bfA}\aH{\bfq} - \aH{\bfb}\right)\:.$$
Let the residual $\aH{\bfr} = \aH{\bfA}\aH{\bfq} - \aH{\bfb}$. By developing the first $n$-rows of the quaternion matrix product $\aH{\bfA}\aH{\bfr}$ we get the conjugated quaternion gradient
\begin{align*}
  \conjgradQ f(\bfq) & = \frac{1}{4} \left(\bfA_1^\herm \bfr + \invol{\left(\bfA_2^\herm \bfr\right)}{\bmi} + \invol{\left(\bfA_2^\herm \bfr\right)}{\bmj} + \invol{\left(\bfA_2^\herm \bfr\right)}{\bmk}\right)\:.
\end{align*}
%
%

\ifCLASSOPTIONcaptionsoff
  \newpage
\fi



%

\renewcommand*{\bibfont}{\footnotesize}
\printbibliography

%

%
%
%
%
%
%
%
%
%
\end{document}